\documentclass[final,1p,times]{elsarticle}
\usepackage{amssymb}
\usepackage{amsmath}
\usepackage{amsthm}
\usepackage[utf8]{inputenc}
\usepackage{marginnote}
\usepackage{amsfonts}
\usepackage{color,soul}
\usepackage{placeins}
\usepackage{makecell}
\usepackage{amssymb}
\usepackage{graphicx}
\usepackage{hyperref}
\usepackage[none]{hyphenat}
\usepackage{lipsum}
\newtheorem{thm}{Theorem}[section]
\newtheorem{lem}[thm]{Lemma}
\newtheorem{rem}[thm]{Remark}
\newtheorem{prop}[thm]{Proposition}
\newtheorem{cor}[thm]{Corollary}
\newtheorem{example}[thm]{Example}
\newdefinition{defn}{Definition}[section]
\numberwithin{equation}{section}
\allowdisplaybreaks
\begin{document}
\begin{frontmatter}
	\title{Optimal dual frames and dual pairs for probability modelled  erasures using weighted average of  operator norm  and spectral radius}
\author{S. Arati}
\ead{aratishashi@iisertvm.ac.in}
\author{P. Devaraj}
\ead{devarajp@iisertvm.ac.in}
\author{Shankhadeep Mondal}
\ead{shankhadeep16@iisertvm.ac.in}

\address{School of Mathematics, Indian Institute of Science Education and Research Thiruvananthapuram, Maruthamala P.O, Vithura, Thiruvananthapuram-695551.}

\begin{abstract}
	The  prime focus of this paper is the study of optimal duals of  a given finite frame as well as   optimal dual pairs,  in the context of probability modelled erasures of frame coefficients. We characterize optimal dual frames (and dual pairs) which, among all dual frames (and dual pairs), minimize the maximum measure of the error operator obtained while considering all possible locations of probabilistic  erasures of frame coefficients in the reconstruction with respect to each dual frame(dual pair). For  a given weight number sequence associated with the probabilities, the measure of the  probabilistic error operator is taken to be the weighted average of the operator norm and the spectral radius. Using this as an optimality measure, the existence and uniqueness of  optimal dual frames (and optimal dual pairs) and their topological properties  are studied. Also, their relations with probabilistic optimal dual frames as well as dual pairs in other contexts, such as those obtained using operator norm and spectral radius as the measure of the error operator, are analyzed.
		
\end{abstract}


\end{frontmatter}

\section{Introduction}
Frames, which were first introduced by Duffin and Schaeffer \cite{duff}, is a generalization of the concept of bases in a Hilbert space. Though the closed linear span of a frame is the whole Hilbert space $\mathcal{H}$, the elements of $\mathcal{H}$ can have multiple and stable representations using the frame elements. Due to this flexibility, reconstruction is possible even if a few frame coefficients are missing. Exploiting this feature of redundancy, frames have been applied in various fields such as sampling and reconstruction, signal processing, communication theory and coding theory. When data is transmitted using frame coefficients, some of them could get corrupted or lost and that introduces an error in the reconstruction, which employs the concept of dual frames. For  a given measure of the error operator, finding an appropriate dual pair of finite frames (or dual frame of a given frame) that minimizes the reconstruction error is a deep rooted problem in frame theory.

\par
One of the initial works on an erasure problem is by Goyal et al. \cite{goya}, wherein they consider a quantization model having additive noise and show that in the case of one erasure, the mean-squared error is minimized by a uniform tight frame. In \cite{casa2}, these equal-norm tight frames have been studied in detail and also particularly in the context of their robustness to erasures. In \cite{holm, bodm1}, the problem of finding a Parseval frame that minimizes the error operator, with respect to the operator norm, over all possible pairs of Parseval frames and their canonical dual frames has been analyzed. For a given frame, the study of dual frames that are optimal for one erasure with respect to the operator norm has been carried out in \cite{jerr,jins}. Certain necessary as well as sufficient conditions for a canonical dual frame to be a unique optimal dual have also been discussed. Taking the spectral radius as the measure of the error operator, the optimal dual frame for a given frame in the case of one erasure has been fully characterized in \cite{sali}. This problem for the case of two erasures has been analysed in \cite{peh, dev}. A similar problem with the measure of the error operator being numerical radius is analysed in \cite{arab} and that  with the average of its operator norm and its numerical radius is studied in \cite{deep}.
\par
In the real world applications, the error patterns are not deterministic, which has led to the consideration of a probability modelled erasure problem. For instance, given a probability distribution for the error locations in the frame expansion, the associated weight number sequence was introduced in \cite{leng3}, using which the error operator is formulated and necessary and sufficient conditions for optimality of  dual frames for a given frame, with respect to the operator norm, is obtained. Furthermore, another set of sufficient conditions, with less complexity, has been given in \cite{li1}.   In the context of finding an optimal dual pair among all possible pairs of Parseval frames and their canonical dual frames, instead of considering the optimal dual frame for a given frame, has been considered in \cite{leng,li}. Recently, the problem of determining  an optimal dual pair among all dual pairs as well as finding an optimal dual frame of a given frame has been analysed in \cite{adm}.
\par
 In this paper, by taking the measure of the probabilistic error operator  to be the weighted  average of the operator norm and the spectral radius, we analyse the erasure problem in the following two broad contexts:
	\begin{itemize}
	\item  The existence and characterization of dual pairs which minimize, among all dual pair of frames, the maximum of the probabilistic error operator's measure obtained by taking all possible combinations of erasure location a fixed length.
	\item The existence and characterization of  dual frames of a given frame, among all dual  frames, which minimize the maximum of  probabilistic error operator's measure by taking all  possible combinations of erasure  locations of a given length.
	\end{itemize}

 For each dual frame  of the given frame and a weight number sequence associated with the probability of erasures, the maximum probabilistic error across various possible locations of fixed length is considered. A dual frame, among all its duals, that minimizes this maximum error is said to be a probabilistic optimal dual frame. In a similar way, an probabilistic optimal dual pair is also defined. The analysis for the deterministic cases of the above stated problems can be obtained from our results,  by setting appropriate values for the weight number sequences.
 After providing the necessary background on frames and notations in connection with probabilistic optimal duality in each of the contexts of dual pair and given frame with respect to  operator norm and spectral radius in Section 2, we introduce the notion of probabilistic  spectral-operator-averaged optimal dual pair and analyze their existence and characterization  in Section 3. The concept of  probabilistic spectral-operator-averaged optimal dual of a given frame is given in Section 4. Necessary as well as sufficient conditions for uncountably many optimal dual frames to exist, for the canonical dual frame to be an optimal dual frame and for it to be the only one are discussed. In Section 5, we investigate the connection between optimality attained under various measures of the error operator, namely the operator norm, the spectral radius and their weighted average. We also show that the canonical dual frame need not be optimal in any of the above contexts, in general. Some topological properties of the collection of the probabilistic spectral-operator-averaged optimal dual frames are studied in Section 6.

\section{Preliminaries on probabilistic erasures }

Let $\mathcal{H}$ denote an $n-$dimensional (real or complex) Hilbert space. A finite sequence of elements $F= \{f_i\}_{i=1}^N$ in $\mathcal{H}$ is said to be  a \textit{frame} for $\mathcal{H}$ if there exist positive  constants $A,B $ such that
\begin{align}
A \left\| f \right\|^2\leq \sum_{i=1}^N\big|\langle f,f_i\rangle\big|^2 \leq B\left\| f \right\|^2,   \forall f\in \mathcal{H}.
\label{frameineq}
\end{align}
The constants $A$ and $B$ appearing inequalities \eqref{frameineq} are called lower and upper frame bounds respectively.  The largest positive constant $A$ satisfying \eqref{frameineq} is called the \textit{optimal lower frame bound}  and the smallest positive constant $B$ satisfying \eqref{frameineq} is called \textit{optimal upper frame bound}. A frame for which both an upper frame bound and a lower frame bound are equal  is called a \textit{tight frame} and if this common  value is one, then the frame is called
 a \textit{Parseval frame}. For a frame $F = \{f_i\}_{i=1}^N$ in $\mathcal{H},$ the linear operator $\Theta_F:  \mathcal{H} \to  {\mathbb{C}^N} $ given by $$\Theta_F(f)= \{\langle f,f_i \rangle\}_{i=1}^N \, ,\,\;f \in \mathcal{H} $$ is known as the \textit{analysis operator} and its adjoint operator   $ \Theta_{F}^*: \mathbb{C}^N \to  \mathcal{H} $  defined by
$$\Theta_{F}^*\left(\{c_i\}_{i=1}^N\right) =  \sum_{i=1}^N {c_i f_i}\, ,\, \; \{c_i\}_{i=1}^N \in \mathbb{C}^{N} $$ is known as the \textit{synthesis operator or preframe operator}.
The operator $S_F : \mathcal{H} \to  \mathcal{H} $  defined by
$$S_{F}f= \Theta_{F}^{*} \Theta_{F} f = \sum_{i=1}^N \langle f,f_i \rangle f_i,\;\, f \in \mathcal{H}$$
is called the \textit{frame operator}, which is a positive, self-adjoint and invertible operator on $ \mathcal{H}.$ For any $f \in \mathcal{H},$ we have the reconstruction formula :
$$f= \sum_{i=1}^N \langle f,S_{F}^{-1}f_i\rangle f_i = \sum_{i=1}^N \langle f,f_i\rangle S_{F}^{-1} f_i .$$
\noindent
A frame $G=\{g_i\}_{i=1}^N$ in $ \mathcal{H}$ is called a \textit{dual frame} of $F$ if every element $f \in   \mathcal{H}$ can be written as
\begin{align}\label{eqn2point1}
	f =  \sum_{i=1}^N \langle f,g_i \rangle f_i .
\end{align}
It may be noted that a sequence $\{g_i\}_{i=1}^N$ satisfies $ f =  \sum\limits_{i=1}^N \langle f,g_i \rangle f_i, \; \forall f \in \mathcal{H}$ if and only if $ f =  \sum\limits_{i=1}^N \langle f,f_i \rangle g_i, \; \forall f \in \mathcal{H}.$ Furthermore, such a sequence $\{g_i\}_{i=1}^N$ will also be a frame.
\noindent
Clearly, $\{S_{F}^{-1}f_i\}_{i=1}^{N}$ is  a dual frame of $F$ and is called the canonical or standard dual frame. If $F$ is a basis, then the canonical dual frame is the only dual frame for it, whereas for a frame which is not a basis, there exist infinitely many dual frames $G$ of $F$ in $ \mathcal{H}.$ A sequence $G= \{g_i\}_{i=1}^N$ is a dual frame of $F$ if and only if it is of the form $G=\{S_{F}^{-1}f_i +u_i\}_{i=1}^N,$ where the sequence $\{u_i\}_{i=1}^N$ satisfies
\begin{align}\label{eqn2point2}
	\sum_{i=1}^N \langle f,f_i \rangle u_i = 0 ,\;\;   \forall f\in \mathcal{H}.
\end{align}
It is easy to see that $\sum\limits_{i=1}^N \langle f,f_i \rangle u_i = 0 ,\;\;   \forall f\in \mathcal{H}$ if and only if $\sum\limits_{i=1}^N \langle f,u_i \rangle f_i = 0 ,\;\;   \forall f\in \mathcal{H}.$
A pair $(F,G)$ of frames consisting of a frame   $F = \{f_i\}_{i=1}^N$ and one of its duals  $G= \{g_i\}_{i=1}^N$ having $N$ elements in each is called an $(N,n)$ dual pair for the n-dimensional Hilbert space ${\cal H}.$
 One may refer to \cite{heil,CasaKuty,ole} for more details on frame theory.
\noindent
\par
In practical data transmissions using  frame coefficients for representing the data, certain frame coefficients may be lost with some probabilities.  Let $P = \{p_i \}_{i=1}^N$ be a probability sequence, where $p_i$ denotes the probability of erasure of the $i^{th}$ frame coefficient $\langle f, f_i \rangle$ in \eqref{eqn2point1}, for $i=1,2,\ldots,N$ and  satisfies
\begin{equation} \label{eqn2.1}
	0 \leq p_i < 1 ,\;\forall \; i=1,2,\ldots,N\; \text{and}\, \displaystyle{\sum_{i=1}^N p_i = 1}.
\end{equation}
Corresponding to each probability distribution $\{p_i\}_{i=1}^N,$
the weight number sequence $\{q_i\}_{i=1}^N$  is defined \cite{leng} as follows:\\
\begin{align}\label{eqn4expressionqi}
	q_i = \displaystyle{{\frac{1}{1 - p_i}\cdot \frac{N-1}{n}}}, \; 1\le i\le N.
\end{align}
It can be easily seen that for $N \geq 2,$ we have $q_i > 0$ and for $N>n,$ we get  $1 \leq q_i < \infty$ for all $i$.
Furthermore, $\sum\limits_{i=1}^N \frac{1}{q_i} = n.$  If $m$ coefficients  have been erased, then the associated error operator is defined by
\begin{align} \label{eqn2point3}
	E_{\Lambda,F,G}\;f:= \Theta_{G}^*D_{P}\Theta_{F} f=\sum_{i\in\Lambda}  q_i \langle f,f_i \rangle g_i\, , \;\, f\in \mathcal{H},
\end{align}
where $\Lambda \subset \{ 1,2,\dots,N \}$ is the set of indices corresponding to the $m$ erased coefficients, $D_p$ is an $N\times N $ diagonal matrix with  diagonal elements $ q_i$ for $i\in \Lambda$ and 0 otherwise.
\subsection{Probabilistic optimal  dual pairs}

For a dual pair $(F, G),$ let
\begin{align*}
\mathcal{O}_{P}^{(1)}(F,G) := \max \bigg\{ \| E_{\Lambda,F,G} \| : |\Lambda|=1 \bigg\}  \mbox{ and } \\
 r_P^{(1)}(F,G):= \max \left\{\; \rho( E_{\Lambda,F,G}) : |\Lambda|=1 \right\},
  \end{align*}
  where $\| E_{\Lambda,F,G} \|$ and $\rho( E_{\Lambda,F,G})$ are the operator norm and the spectral radius of $E_{\Lambda,F,G},$ and the maximum is taken over all subsets $\Lambda$ of $\{1,2,\ldots,N\}$ with cardinality $1.$
We define
$$ \mathcal{O}_P^{(1)} := \inf \bigg\{ \mathcal{O}_P^{(1)}(F,G) : (F,G) \text{ is an (N,n) dual pair } \bigg\}, $$ and
$$ r_P^{(1)} := \inf \bigg\{ r_P^{(1)}(F,G) : (F,G) \text{ is an (N,n) dual pair } \bigg\}. $$
A dual pair $(F',G')$ is  called a $1$-erasure probabilistic optimal dual (in short $POD$ pair) if
$\mathcal{O}_P^{(1)}(F',G') = \mathcal{O}_P^{(1)}, $
and a dual pair $(F',G')$ is  called an $1$-erasure probabilistic spectrally optimal dual pair (in short $PSOD$ pair) if
$ r_P^{(1)}(F',G') =    r_P^{(1)}.$

\subsection{Probabilistic optimal  dual of a given frame}
A dual frame $G'$ of $F$ is called a $1-$erasure probabilistic optimal dual frame (in short $POD$) of $F$ if
$$\mathcal{O}_P^{(1)}(F,G') = \inf \left\{ \mathcal{O}_P^{(1)}(F,G)  : G \;\text{is a dual of F} \right\} .$$

 A dual frame $G'$ of $F$ is called a $1-$erasure probabilistic spectrally optimal dual frame (in short $PSOD$) of $F$ if
$$ r_P^{(1)}(F,G') =   \inf \left\{ r_P^{(1)}(F,G) : G \;\text{is a dual of } F \right\} .$$
The set of all  $1-$erasure $PSODs$ of $F$ shall be denoted by $PSOD^{(1)}(F).$

\section{ Analysis of $PA_{SO}OD$  pairs}
In this section, we introduce a new probabilistic optimality measure, which uses the weighted average of the operator norm and the spectral radius of the probabilistic error operator. Using this optimality  measure,  we analyse the existence and characterization of a optimal dual pairs.
In our study, we can take the weight number sequence to be any sequence $\{q_i\}_{i=1}^N$ of real numbers satisfying the following conditions:
\begin{enumerate}
\item  $q_i \ge 1,$  $1\le i\le N,$
 \item $\sum\limits_{i=1}^N \frac{1}{q_i} = n.$
\end{enumerate}

First, we shall define the concept of probabilistic spectral-operator-averaged optimal dual pairs. Let $(F,G) $ be an $(N,n)$ dual pair for an $n-$dimensional Hilbert space $\mathcal{H}$. Let $E_{\Lambda,F,G}$ be  the error operator, as defined in \eqref{eqn2point3}, associated with $m$ erasures. For $ 0 \leq \lambda \leq 1,$ we shall denote the corresponding maximum probabilistic spectral-operator-averaged error by $ \mathcal{A}_{\lambda P}^{(m)}(F,G).$  In other words,
\begin{align*} \label{eqn3point2}
	\mathcal{A}_{\lambda P}^{(m)}(F,G) = \max \bigg\{ \lambda \rho{( E_{\Lambda,F,G})} + (1- \lambda)\left\|  E_{\Lambda,F,G}  \right\|: | \Lambda | = m \bigg\},
\end{align*}
where the maximum is taken over all subsets $\Lambda$ of $\{1,2,\ldots,N\}$ with cardinality $m,\,\rho(.)$ denotes the spectral radius and  $\left\| .  \right\| $ denotes the operator norm.
Now, we define
\begin{align*}
	\mathcal{A}_{\lambda P}^{(1)} := \inf \left\{ \mathcal{A}_{\lambda P}^{(1)}(F,G) : (F,G) \, \text{is  an $(N,n)$ dual pair}  \right\}
\end{align*}
A  dual pair $(F,G)$ is called a $1-$erasure Probabilistic Spectral-Operator Averaged Optimal Dual pair (in short $PA_{SO}OD$ pair) if $\mathcal{A}_{\lambda P}^{(1)}(F,G)  = \mathcal{A}_{\lambda P}^{(1)}. $
The set of all  $1$-erasure  $PA_{SO}OD$ pair for  $\mathcal{H}$  is denoted by
$\displaystyle{ \zeta_{\lambda P}^{(1)}}.$
 The following proposition gives an equivalent expression for $\mathcal{A}_{\lambda  P}^{(1)}(F,G),$ which we shall make use of in the sequel of the paper.

\begin{prop} \label{prop3point1}
	Let  $(F, G)$ be an $(N,n)$ dual pair for ${\cal H}$ and let $\{q_i\}_{i=1}^N $ be  any sequence of positive real numbers. Then, the following holds:
	$$ \mathcal{A}_{\lambda P}^{(1)}(F,G) = \max_{1 \leq i \leq N} \;\; q_i\bigg( \lambda | \langle f_i,g_i \rangle | + (1- \lambda)\|f_i\|\; \|g_i\|\bigg),\, 0\le \lambda\le 1. $$
	
\end{prop}

\begin{proof}
	In the case of 1-erasure, if the erasure occurs in the $i^{th}$ position, then the error operator is $E_{\Lambda,F,G}\,f = q_i \langle f,f_i \rangle g_i,\;f \in \mathcal{H}.$ It can be verified that the eigenvalues of $E_{\Lambda,F,G}$ are $0 \; \text{and} \;q_i \langle g_i,f_i \rangle.$ Therefore, $\rho(E_{\Lambda,F,G}) = |q_i \langle f_i,g_i \rangle |.$ Further,
	\begin{eqnarray} \label{equationON5}
		\| E_{\Lambda,F,G}\| = \sup\limits_{\|f\|=1} \left\| E_{\Lambda,F,G}\,f   \right\|=  \sup\limits_{\|f\|=1} \left\|  q_i \langle f,f_i \rangle g_i  \right\| = q_i\|f_i\|\; \|g_i\|.
	\end{eqnarray}
	Hence, $\mathcal{A}_{\lambda P}^{(1)}(F,G) =  \max\limits_{1 \leq i \leq N} \;\; q_i\bigg( \lambda | \langle f_i,g_i \rangle | + (1- \lambda)\|f_i\|\; \|g_i\|\bigg) ,$ thereby proving the proposition.
\end{proof}

In order to analyse the existence of optimal dual pairs, we need the concept of probabilistic equal norm Parseval frames and existence of  such frames.

\begin{defn}
 A  frame $\{f_i\}_{i=1}^N $ is said to be a \textit{probabilistic equal norm  frame} if there exists a constant $c$ such that
	$\sqrt{q_i} \|f_i\| = c$  for all $1 \leq i \leq N.$ A  frame is called a \textit{probabilistic equal norm Parseval frame} if it is both a Parseval frame and a probabilistic equal norm frame.
\end{defn}
It may be  noted that if $\{f_i\}_{i=1}^N $ is a probabilistic equal norm Parseval frame, then the constant  appearing in the above definition turns out to be 1.
\begin{prop}\label{cor4point1}
For a given $n-$ dimensional Hilbert space $\mathcal{H} $ and a weight number sequence $\{q_i\}_{i=1}^N ,$  there  exists a probabilistic equal norm  Parseval frame $F = \{f_i\}_{i=1}^N$ for  $\mathcal{H} .$
\end{prop}
\begin{proof}
 Let $S$ denote the $n \times n$ identity matrix. Then, the  eigenvalues of $S$ are $\lambda_i =1, \;\; 1\leq i \leq n.$ Now, we permute the weight number sequence $\{q_i\}_{i=1}^N$ and obtain a new weight number sequence $\{q'_i\}_{i=1}^N$ such that $q'_1 \leq q'_2 \leq \dots \leq q'_N. $
Let $ a'_i = \frac{1}{\sqrt{q'_i}}, \, 1 \leq i \leq N. $
	Then, for $1 \leq k \leq n,$  $\displaystyle{\sum_{i=1}^k a_{i}'^2  \leq k = \sum_{i=1}^k \lambda_i }$ and $\displaystyle{\sum_{i=1}^N a_{i}'^2 = \sum_{i=1}^N \frac{1}{q'_i} = n =  \sum_{i=1}^n \lambda_i.} $
	Therefore, by \cite[Theorem 2.1]{cass2}, there exists a frame $\{f'_i \}_{i=1}^N$ for $\mathcal{H}$  with $S$ as its frame operator and $\|f'_i \| = \frac{1}{\sqrt{q'_i}}.$ Thus, $\{f'_i \}_{i=1}^N$ is a probabilistic equal norm Parseval frame. Applying the inverse permutation, we obtain a probabilistic equal norm Parseval frame $F = \{f_i \}_{i=1}^N$ with $\|f_i \| =  \frac{1}{\sqrt{q_i}}, \, 1 \leq i \leq N.$
\end{proof}

\begin{thm}\label{thm6point4} Let   $\{q_i\}_{i=1}^N $ be  a weight number sequence. Then, the following hold:
\begin{enumerate}
\item  $ \mathcal{A}_{\lambda P}^{(1)}=1$ for  every $0 \leq \lambda \leq 1.$
\item For $0 < \lambda < 1,$  $(F,G) \in \zeta_{\lambda P}^{(1)} $ if and only if $ \langle f_i,g_i \rangle = \|f_i\|\;\|g_i\| = \frac{1}{q_i},\, 1 \leq i \leq N.$
\item For $\lambda=0, $  $(F,G) \in \zeta_{\lambda P}^{(1)} $ if and only if $  \|f_i\|\;\|g_i\| = \frac{1}{q_i},\, 1 \leq i \leq N.$
\item For $\lambda=1, $  $(F,G) \in \zeta_{\lambda P}^{(1)} $ if and only if $ \langle f_i,g_i \rangle  = \frac{1}{q_i}, \, 1 \leq i \leq N.$
    \end{enumerate}
\end{thm}
\begin{proof}
	1. Let $(F,G)$ be an $(N,n)$ dual pair. Then, by using  Proposition \ref{prop3point1}, we get $$ \mathcal{A}_{\lambda P}^{(1)}(F,G) = \max\limits_{1 \leq i \leq N} \, q_i\bigg\{ \lambda | \langle f_i,g_i \rangle | + (1 - \lambda)\|f_i\|\; \|g_i\|\bigg\}  \geq \max\limits_{1 \leq i \leq N}\, q_i|\langle f_i,g_i \rangle| .$$ It is easy to see  that  $ \mathcal{A}_{\lambda P}^{(1)}(F,G) \ge 1,$ for otherwise $ \mathcal{A}_{\lambda P}^{(1)}(F,G) < 1.$ Then, $q_i |\langle f_i,g_i \rangle|  < 1,\; \forall i.$ This leads to
	$ n = \sum\limits_{1 \leq i \leq N}  \langle f_i,g_i \rangle \leq  \sum\limits_{1 \leq i \leq N} |  \langle f_i,g_i \rangle | < \sum\limits_{1 \leq i \leq N} \frac{1}{q_i} =n,$
	which is not possible. Thus, $ \mathcal{A}_{\lambda P}^{(1)}(F,G) \geq 1$ for every dual pair $(F,G).$  Now, if we take a dual pair $(F,F)$ consisting of  probabilistic equal norm Parseval frame  and its canonical dual, whose existence is obtained  in Proposition \ref{cor4point1}, we get $\mathcal{A}_{\lambda P}^{(1)}(F,F) = \max\limits_{1 \leq i \leq N}q_i\bigg\{\lambda | \langle f_i,f_i \rangle | + (1- \lambda)\|f_i\|\,\|f_i\|\bigg\} = \max\limits_{1 \leq i \leq N} q_i \|f_i\|^2 = 1.$ Therefore, $\mathcal{A}_{\lambda P}^{(1)} = 1.$\\
2.  Let $(F,G) \in \zeta_{\lambda P}^{(1)}.$ Then, $\max\limits_{1 \leq i \leq N} q_i\bigg\{\lambda | \langle f_i,g_i \rangle | + (1- \lambda)\|f_i\|\,\|g_i\|\bigg\} = 1.$ If  $ q_j\bigg\{\lambda | \langle f_j,g_j \rangle | + (1- \lambda)\|f_j\|\,\|g_j\|\bigg\} < 1$ for   $j \in
	\{1,2,\ldots,N\},$ then
	$$n = \sum\limits_{1 \leq i \leq N}  \langle f_i,g_i \rangle \leq  \sum\limits_{1 \leq i \leq N} |  \langle f_i,g_i \rangle | \leq  \sum\limits_{1 \leq i \leq N} \bigg\{\lambda | \langle f_i,g_i \rangle | + (1- \lambda)\|f_i\|\,\|g_i\| \bigg\}    < \sum\limits_{1 \leq i \leq N} \frac{1}{q_i} =n, $$
	which is not possible. Therefore, $ \lambda | \langle f_i,g_i \rangle | + (1- \lambda)\|f_i\|\,\|g_i\| = \frac{1}{q_i}, \; 1 \leq i \leq N.$ Now, we shall prove that $| \langle f_i,g_i \rangle |= \frac{1}{q_i},\;\forall i.$ Obviously, $| \langle f_i,g_i \rangle | > \frac{1}{q_i}$ is not possible for any $i,$ otherwise,
	$$q_i \bigg\{\lambda | \langle f_i,g_i \rangle | + (1- \lambda)\|f_i\|\,\|g_i\| \bigg\}  \geq q_i| \langle f_i,g_i \rangle | >1, $$
	which contradict the fact that  $\lambda | \langle f_i,g_i \rangle | + (1- \lambda)\|f_i\|\,\|g_i\|  = \frac{1}{q_i}.$ Therefore, $| \langle f_i,g_i \rangle | \leq \frac{1}{q_i},\;1 \leq i \leq N.$ If for any $j\in \{1,2,\ldots,N\},$\, $| \langle f_j,g_j \rangle |< \frac{1}{q_j},$ then $n = \sum\limits_{1 \leq i \leq N}  \langle f_i,g_i \rangle \leq  \sum\limits_{1 \leq i \leq N} |  \langle f_i,g_i \rangle | < \sum\limits_{1 \leq i \leq N} \frac{1}{q_i} =n, $ which is not possible. Therefore, $| \langle f_i,g_i \rangle |= \frac{1}{q_i},\;\forall i.$ Also, using the fact that $\lambda| \langle f_i,g_i \rangle | + (1- \lambda) \|f_i\|\,\|g_i\|  = \frac{1}{q_i},\forall i,$ we have $\|f_i\|\,\|g_i\| = \frac{1}{q_i},\,\forall i.$
It now suffices to show that $\langle f_i, g_i \rangle \geq 0,\,\forall\, i.$ Clearly,
	$$\sum\limits_{i=1}^N \langle f_i , g_i \rangle = n = \sum\limits_{i=1}^N\frac{1}{q_i}= \sum\limits_{i=1}^N |\langle f_i , g_i \rangle | .$$
	Let $\langle f_j , g_j \rangle = a_j + ib_j$ for  $1\leq j \leq N,$ where $i = \sqrt{-1}$ and $a_j,b_j \in \mathbb{R}.$ Then, $\sum\limits_{i=1}^N a_i = n = \sum\limits_{i=1}^N \sqrt{a_i^2 +b_i^2,}$ which implies $b_i = 0$ for every $i$ and hence $\sum\limits_{i=1}^N a_i = n = \sum\limits_{i=1}^N |a_i|.$ This shows that  $a_i \geq 0,\, \forall\, i.$ Therefore, $\langle f_i , g_i \rangle = \frac{1}{q_i}$ for $1 \leq i \leq  N.$

	Conversely, suppose  $ \langle f_i,g_i \rangle = \|f_i\|\;\|g_i\| = \dfrac{1}{q_i}, \;1 \leq i \leq N.$ Then, $\lambda | \langle f_i,g_i \rangle | + (1- \lambda)\|f_i\|\,\|g_i\|  = \frac{1}{q_i},\;\forall i $ and hence, $  \mathcal{A}_{\lambda P}^{(1)}(F,G) = 1,$ by Proposition \ref{prop3point1}.  Therefore,  using the statement 1, we have $(F,G) \in \zeta_{\lambda P}^{(1)}.$

The proof of the other statements are similar.
\end{proof}

\begin{cor}Let $0<\lambda\le 1.$ Then, every $1-$erasure $PA_{SO}OD$ pair is  probabilistic $1-$uniform.
\end{cor}
\begin{proof}
Follows from the proof of Theorem \ref{thm6point4}.
\end{proof}

It is simple to verify that $(F,G)$ is an $(N,n)$ dual pair if and only if $(UF,UG)$ is an $(N,n)$ dual pair for every unitary operator $U$ on $\mathcal{H}.$ The following proposition shows that  an $1-$erasure optimal  pair invariant under the action a unitary operator.
\begin{prop}
	Let $(F,G)$ be an $(N,n)$ dual pair in $\mathcal{H}$ and let $\{q_i\}_{i=1}^N$ be a weight number sequence.  Let  $0 \leq \lambda \leq 1.$ Then, for any unitary operator $U $ of $\mathcal{H},$  $(F,G) \in \zeta_{\lambda P}^{(1)} $ if and only if $(UF,UG) \in \zeta_{\lambda P}^{(1)}.$
\end{prop}
\begin{proof}
	The proof follows from
	\begin{align*}
		\mathcal{A}_{\lambda P}^{(1)}  (UF,UG) &= \max\limits_{1 \leq i \leq N} q_i \bigg\{\lambda | \langle U f_i, U g_i \rangle | + (1- \lambda)\|U f_i\|\,\|U g_i\| \bigg\}  \\&= \max\limits_{1 \leq i \leq N} q_i \bigg\{\lambda | \langle f_i, g_i \rangle | + (1- \lambda)\| f_i\|\,\|g_i\| \bigg\} \\&=\mathcal{A}_{\lambda P}^{(1)} (F,G). \end{align*}
\end{proof}

\begin{thm}\label{thm4point2}
	Let $F = \{f_i\}_{i=1}^N $ be a tight frame for $\mathcal{H}$. Let  $\{q_i\}_{i=1}^N$ be a weight number sequence and  $0 \leq \lambda \leq 1.$  Then, the following are equivalent:
	\begin{enumerate}
		\item [{\em (i)}] $(F,S_{F}^{-1}F)$ is a $1-$erasure $POD$ pair.
		\item [{\em (ii)}] $(F,S_{F}^{-1}F)$ is a $1-$erasure $PSOD$ pair.
		\item [{\em (iii)}] $(F,S_{F}^{-1}F)$ is a $1-$erasure $PA_{SO}OD$ pair.
	\end{enumerate}
\end{thm}

\begin{proof}
\noindent $(i) \Rightarrow (ii): $ \;$(F,S_{F}^{-1}F)$ is a $1-$erasure $POD$ pair gives $\frac{q_i}{A} \|f_i\|^2 = 1$ for every $i,$ by \cite[Theorem 3.4]{adm}. Therefore,  $r_{P}^{(1)}(F,S_{F}^{-1}F) = \max\limits_{1 \leq i \leq N} \frac{q_i}{A} \|f_i\|^2 = 1 $ and hence  $(F,S_{F}^{-1}F)$ is a $1-$erasure $PSOD$ pair by  \cite[Theorem 3.8]{adm}.

\noindent $(ii) \Rightarrow (iii) $:  \;$(F,S_{F}^{-1}F)$ is a $1-$erasure $PSOD$ pair gives $\frac{q_i}{A} \|f_i\|^2 = 1$ for every $i$, by  \cite[Theorem 3.8]{adm}. Therefore, $\mathcal{A}_{\lambda P}^{(1)} (F,S_{F}^{-1}F) = \max\limits_{1 \leq i \leq N} \frac{q_i}{A} \|f_i\|^2 = 1. $ Hence, $(F,S_{F}^{-1}F)$ is a $1-$erasure $PA_{SO}OD$ pair, by Theorem \ref{thm6point4}.
	
\noindent $(iii) \Rightarrow (i) $: The proof is similar to that of  $(ii) \Rightarrow (iii) $.
	
\end{proof}

\begin{rem}
If  $q_i=\frac{N}{n},$  $\forall i,$ in the weight number sequence  $\{q_i\}_{i=1}^N,$ then all the results in this section correspond to the usual average case.
\end{rem}

\section{Analysis of $PA_{SO}OD$ of a given frame}
In this section, we analyse optimal duals  of a given frame using the weighted average of the operator norm and spectral radius of the error operator as the measure of optimality. We can take the weight number sequence  $\{q_i\}_{i=1}^N $ to be any sequence of positive real numbers. We define
\begin{align*}
	\mathcal{A}_{\lambda P}^{(1)}(F) :=  inf \left\{ \mathcal{A}_{\lambda P}^{(1)}(F,G) : \text{$G$ is a dual of $F$} \right\}.
\end{align*}
A dual frame $G'$ of $F$ is called $1-$erasure  $PA_{SO}OD$ frame of $F$ if $\mathcal{A}_{\lambda P}^{(1)}(F,G') = \mathcal{A}_{\lambda P}^{(1)}(F) .$
 We denote the set of all $1-$erasure $PA_{SO}OD$s of $F$ by $\Delta_{F}^{(1)},$ for the sake of convenience. When $F$ is a basis, it is obvious that $\Delta_{F}^{(1)}= \left\{S_{F}^{-1}F\right\}.$ In this section, we consider the question of existence and uniqueness of a $1-$erasure $PA_{SO}OD $ of a given frame, which need not be a  basis.
\noindent
A few sufficient conditions for a non-empty $\Delta_{F}^{(1)}$ are discussed in the theorem below, which makes use of some notations given as follows.\\
 For a frame $F = \{f_i\}_{i=1}^N $ in $\mathcal{H},$ let $L_{\lambda}= \max \limits_{1 \leq i \leq N}  q_i\left(\lambda\left\| S_{F}^{-1/2}f_i \right\|^2 + (1-\lambda)\|f_i\|\;\left\| S_{F}^{-1}f_i\right\| \right)  $ and $\Lambda_1, \Lambda_2$ be subsets of $\{1,2,\ldots,N\}$ such that  $\Lambda_1 = \left\{ i : q_i \left( \lambda\left\| S_{F}^{-1/2}f_i \right\|^2 + (1-\lambda)\|f_i\|\;\left\| S_{F}^{-1}f_i\right\| \right) = L_{\lambda}  \right\} $ and $\Lambda_2 = \{1,2,\ldots, N\} \setminus \Lambda_1.$ We denote  $\text{span} \left\{f_i : i \in \Lambda_j  \right\}$ by $H_j,$ for $j=1,2.$\\

\begin{thm} \label{thm3point8}
	Let $F = \{f_i\}_{i=1}^N $ be a frame for $\mathcal{H}$ and $H_1,H_2,\Lambda_{1}$ be as defined above. Let  $\{q_i\}_{i=1}^N $ be a weight number sequence and  $0\le \lambda\le 1.$ If $H_1 \cap H_2 = \{0\}$ and $\{f_i\}_{i\in \Lambda_1}$  is linearly independent, then $S_{F}^{-1} F \in \Delta_{F}^{(1)}.$ Moreover, if $N=n,$ then $\Delta_{F}^{(1)} = \{ S_{F}^{-1} F \}$  and  if $N > n,$ then $\{S_{F}^{-1} F \} \subsetneq \Delta_{F}^{(1)}.$ In fact, $\Delta_{F}^{(1)}$ is uncountable in this case.
\end{thm}

\begin{proof}
	Let $G$ be any dual of $F.$ Then, as mentioned in \eqref{eqn2point2}, $G$ can be written as $G = \left\{ S_{F}^{-1}f_i + u_i \right\}_{i=1}^N,$ where $\sum\limits_{i=1}^N \langle f,u_i \rangle f_i =0, \; f \in \mathcal{H}.$ This in turn implies that  $\sum\limits_{i \in \Lambda_1} \langle f,u_i \rangle f_i =0 = \sum\limits_{i \in \Lambda_2} \langle f,u_i \rangle f_i, $ as $H_1 \cap H_2 = \{0\}.$ The linear independence of $\{f_i : i \in \Lambda_1\}$ then gives $\langle f,u_i \rangle = 0,$ for all $i \in \Lambda_1, f \in \mathcal{H}.$ Therefore, $u_i =0,$ for all $i \in \Lambda_1.$	Now, by Proposition \ref{prop3point1},
	\begin{align*}
		\mathcal{A}_{\lambda P}^{(1)}(F,G) &=  \max\limits_{1 \leq i \leq N} q_i\left(\lambda \left|\langle f_i,S_{F}^{-1} f_{i} + u_i \rangle \right| + (1-\lambda) \|f_i\| \left\|S_{F}^{-1} f_{i} + u_i\right\|\right)  \\& \geq \max\limits_{i \in \Lambda_1}q_i\left(\lambda \left|\langle f_i,S_{F}^{-1} f_{i} + u_i \rangle \right| + (1-\lambda) \|f_i\| \left\|S_{F}^{-1} f_{i} + u_i\right\|\right)   \\&= \max\limits_{i \in \Lambda_1} q_i \left(\lambda \left\| S_{F}^{-1/2} f_{i} \right\|^2 + (1-\lambda)\|f_i\| \left\|S_{F}^{-1} f_{i} \right\| \right) \\&= L_{\lambda} \\&= \max\limits_{1 \leq i \leq N}  q_i\left(\lambda \left|\langle f_i,S_{F}^{-1} f_{i} + u_i \rangle \right| + (1-\lambda) \|f_i\| \left\|S_{F}^{-1} f_{i} + u_i\right\|\right)  \\&= \mathcal{A}_{\lambda p}^{(1)}(F,S_{F}^{-1}F ).
	\end{align*}
	Hence, the canonical dual is a 1-erasure $PA_{SO}OD$ of $F.$
	\par If $N=n,$ then $S_{F}^{-1} F$ is the only dual of $F.$ Hence, $\Delta_{F}^{(1)} = \{ S_{F}^{-1} F \}.$ Suppose $N >n.$ Then, there exists a dual $G$ of $F$ other than the canonical dual. In other words, there exists a dual $G = \{g_i\}_{i=1}^N = \left\{S_{F}^{-1} f_{i} + u_i \right\}_{i=1}^N$ of $F$ such that $u_i =0$ for all $i \in \Lambda_1$ and $u_{i_0} \neq 0$ for some $i_0 \in \Lambda_2.$ As $t \mapsto $ $q_i\left(\lambda|\langle f_i, S_{F}^{-1} f_{i} + t u_i \rangle | + (1-\lambda)\|f_i \| \left\| S_{F}^{-1} f_{i} + t u_i \right\|\right)$ is a continuous function on $\mathbb{R}$ and  $q_i\left(\lambda|\langle f_i, S_{F}^{-1} f_{i}  \rangle | + (1-\lambda)\|f_i \| \left\| S_{F}^{-1} f_{i}  \right\|\right)<L_{\lambda},$ for all $i \in \Lambda_2,$ there exists a $\delta > 0 $  such that for any $0 \leq |t| \leq \delta,$  $q_i\left(\lambda|\langle f_i, S_{F}^{-1} f_{i} + t u_i \rangle | + (1-\lambda)\|f_i \| \left\| S_{F}^{-1} f_{i} + t u_i \right\|\right) < L_{\lambda}, $ for all $i \in \Lambda_2.$ Let $|t| \leq \delta.$ Clearly, $G_t = \{ S_{F}^{-1} f_{i} + t u_i\}_{i=1}^N $
	is a dual of $F.$
	Further,
	\begin{align*}
		\mathcal{A}_{\lambda P}^{(1)}(F,G_t) &=  \max\limits_{1 \leq i \leq N} q_i\left(\lambda|\langle f_i, S_{F}^{-1} f_{i} + t u_i \rangle | + (1-\lambda)\|f_i \| \left\| S_{F}^{-1} f_{i} + t u_i \right\|\right)  \\& = L_{\lambda} \\& =  \max\limits_{1 \leq i \leq N} q_i\left(\lambda|\langle f_i, S_{F}^{-1} f_{i}  \rangle | + (1-\lambda)\|f_i \| \left\| S_{F}^{-1} f_{i}  \right\|\right)\\&= \mathcal{A}_{\lambda P}^{(1)}(F,S_{F}^{-1}F) \\&= \mathcal{A}_{\lambda P}^{(1)}(F),
	\end{align*}
	thereby showing that $G_t = \{ S_{F}^{-1} f_{i} + t u_i\}_{i=1}^N \in \Delta_{F}^{(1)},$ for $|t| \leq \delta.$
\end{proof}
\noindent It may be noted however that the linear independence of $\{f_i\}_{i \in \Lambda_1}$ is not necessary for $S_{F}^{-1} F \in \Delta_{F}^{(1)},$ as illustrated by the example below.

\begin{example}
	Let $ \mathcal{H}= {\mathbb{R}}^2$,\,$F= \{f_1,f_2,f_3\} \subset \mathcal{H},$ where   $f_1 = \left[\begin{array}{l}
		1 \\0
	\end{array}\right] $,
	$f_2 = \left[\begin{array}{l}
		0 \\ \frac{1}{2}
	\end{array}\right] $,
	$ f_3 = \left[\begin{array}{l}
		0 \\ \frac{1}{2}
	\end{array}\right] $ and  let  $\{q_i\}_{i=1}^3 = \{1,2,2\}$ be the weight number sequence and  $0\le \lambda < 1.$
	\par The above sequence $F$ is a frame with bounds $\frac{1}{2}$ and $1.$ 	It can be verified that the canonical dual of $F$ is
	$S_{F}^{-1} F = \left\{ \left[\begin{array}{r} 1 \\ 0 \end{array}\right],  \left[\begin{array}{r} 0 \\ 1 \end{array}\right], \left[\begin{array}{l} 0 \\ 1  \end{array}\right] \;\right\}. $ Clearly,
	$$L_{\lambda}= \max \left\{ q_i\left(\lambda\left\| S_{F}^{-1/2}f_i \right\|^2 + (1-\lambda)\|f_i\|\;\left\| S_{F}^{-1}f_i\right\| \right) : 1 \leq i \leq 3 \right\} = \max \left\{1, 1, 1\right\} = 1.$$
	 So, $\Lambda_1 = \{1,2,3\},\Lambda_2 = \phi.$ $H_1 = \left\{\left[\begin{array}{l}
		a \\ b
	\end{array}\right]: a,b \in \mathbb{R}\right\} \text{ and } H_2 = \left\{\left[\begin{array}{l}
		 0\\ 0
	\end{array}\right]\right\}.$ Obviously, $H_1 \cap H_2 =\{0\}$ and  $\{f_i\}_{i \in \Lambda_1}$ is linearly dependent. It is known that any dual of $F$ is of the form $G= \{S_F^{-1}f_i + u_i\}_{1\leq i \leq 3},$ where  $\sum\limits_{1 \leq i \leq 3} \langle f,f_i \rangle u_i =0,\forall f\in \mathbb{R}^2$ and  it can be shown that $\{u_i\}_{1 \leq i \leq 3} =  \left\{\left[\begin{array}{l}
		0 \\ 0
	\end{array}\right], \left[\begin{array}{l}
		\alpha \\ \beta
	\end{array}\right], \left[\begin{array}{l}
		-\alpha \\ -\beta
	\end{array}\right]\right\}$, $ \alpha,\,\beta \in \mathbb{R}.$ Therefore, all the duals are of the form $G= \left\{ \left[\begin{array}{r} 1 \\ 0 \end{array}\right],  \left[\begin{array}{r} \alpha \\ 1 + \beta \end{array}\right], \left[\begin{array}{l} -\alpha \\ 1- \beta  \end{array}\right] \;\right\},$ where $\alpha,\,\beta \in \mathbb{R}.$ Now,  $\mathcal{A}_{\lambda P}^{(1)}(F,S_{F}^{-1}F) = L_{\lambda}= 1$ and $$\mathcal{A}_{\lambda P}^{(1)}(F,G) =  \max \left\{ 1,\, \lambda\left| 1+\beta \right| + (1-\lambda)\sqrt{\alpha^2 + (1+\beta)^2},\,  \lambda\left| 1-\beta \right| + (1-\lambda)\sqrt{\alpha^2 + (1-\beta)^2} \right\}.$$ For $\beta >0,$ we have $\lambda\left| 1+\beta \right| + (1-\lambda)\sqrt{\alpha^2 + (1+\beta)^2} > 1$ and so the corresponding $G$ has $\mathcal{A}_{\lambda P}^{(1)}(F,G) > 1.$ Likewise, for $\beta <0,$   $\lambda\left| 1-\beta \right| + (1-\lambda)\sqrt{\alpha^2 + (1-\beta)^2} > 1$ and $\mathcal{A}_{\lambda P}^{(1)}(F,G) > 1.$ For $\beta =0,$ $\mathcal{A}_{\lambda P}^{(1)}(F,G) =  \max \left\{ 1,\, \lambda + (1-\lambda)\sqrt{\alpha^2 + 1} \right\}> 1 ,$ when $\alpha \neq 0.$ This implies that $\mathcal{A}_{\lambda P}^{(1)}(F) = 1$ and hence, $\Delta_{F}^{(1)} = \left\{S_F^{-1}F\right\}.$
	
\end{example}
\begin{rem}
	The above example also shows that the condition on the linear independence of $\{f_i\}_{i \in \Lambda_1}$ cannot be dropped for having $\Delta_{F}^{(1)}$ uncountable when $N>n.$
\end{rem}

The following result provides some other sufficient conditions under which $\Delta_{F}^{(1)}$ is not empty for a tight frame $F$. In fact, it is unique in this case.

\begin{thm}\label{thm3point7}
	Let $F = \{f_i\}_{i=1}^N $ be a tight frame for $\mathcal{H}.$  Let  $\{q_i\}_{i=1}^N $ be a weight number sequence and $0\le \lambda < 1.$ If there exists a constant $c$ such that  $q_i\|f_i \|^2 = c$  for all $1 \leq i \leq N,$  then $\Delta_{F}^{(1)} = \left\{ S_{F}^{-1}F \right\}. $
\end{thm}

\begin{proof}
	Let $F = \{f_i\}_{i=1}^N$ be a tight frame with bound $A.$ Suppose  $ S_{F}^{-1}F  \notin \Delta_{F}^{(1)}.$ Then, $F$ cannot be a  basis and there exists a dual $G = \left\{ S_{F}^{-1}f_i + u_i \right\}_{i=1}^N = \left\{\frac{1}{A}f_i + u_i\right\}_{i=1}^N $ of $F,$ with $ \sum \limits_{i=1}^N \langle f, u_i \rangle f_i =0,\;\forall f \in \mathcal{H},$  such that $\mathcal{A}_{\lambda P}^{(1)}(F,G) < \mathcal{A}_{\lambda P}^{(1)}(F,S_{F}^{-1}F).$ Then, for $1 \leq i \leq N,$
	
	\begin{align*}
		 q_i\bigg\{ \lambda|\langle f_i, \frac{1}{A}f_{i} + u_i\rangle | + (1-\lambda) \|f_i\|\; \left\|\frac{1}{A} f_{i} +u_i\right\| \bigg\} &\leq \max\limits_{1 \leq i \leq N} q_i\bigg\{ \lambda|\langle f_i, \frac{1}{A}f_{i} + u_i\rangle | + (1-\lambda) \|f_i\|\; \left\|\frac{1}{A} f_{i} +u_i\right\| \bigg\} \\&< \max\limits_{1 \leq i \leq N} \frac{q_i \|f_i\|^2}{A}  = \frac{c}{A}.
	\end{align*}
	Also,
	\begin{align*}
	&q_i\bigg\{ \lambda|\langle f_i, \frac{1}{A}f_{i} + u_i\rangle | + (1-\lambda) \|f_i\|\; \left\|\frac{1}{A} f_{i} +u_i\right\| \bigg\} \\ &= \lambda \left| \frac{q_i}{A}\|f_i\|^2 + q_i\langle f_i, u_i\rangle \right| +  (1-\lambda)\sqrt{\frac{q^2_i}{A^2}\|f_i\|^4 + q^2_i \|f_i\|^2 \,\|u_i\|^2 + \frac{2q^2_i \|f_i\|^2}{A} Re \langle f_i, u_i\rangle}  \\&=\lambda \left| \frac{c}{A} + q_i\langle f_i, u_i\rangle \right| + (1-\lambda) \sqrt{\frac{c^2}{A^2} + cq_i\|u_i\|^2 + \frac{2cq_i}{A} Re \langle f_i, u_i\rangle} .
	\end{align*}
	So, $$\lambda \left| \frac{c}{A} + q_i\langle f_i, u_i\rangle \right| + (1-\lambda) \sqrt{\frac{c^2}{A^2} + cq_i\|u_i\|^2 + \frac{2cq_i}{A} Re \langle f_i, u_i\rangle} < \frac{c}{A},\; \text{for all}\; 1\leq i \leq N.$$ Now, if $Re \langle f_i,u_i \rangle \geq 0,$ then
	\begin{align*} &\left| \frac{c}{A} + q_i\langle f_i, u_i\rangle \right| \geq \frac{c}{A} + q_i Re \langle f_i,u_i \rangle \geq \frac{c}{A}
		\quad \text{and} \quad \frac{c^2}{A^2} + cq_i\|u_i\|^2 + \frac{2cq_i}{A} Re \langle f_i, u_i\rangle \geq  \frac{c^2}{A^2}.
	\end{align*}
	This implies that $\lambda \left| \frac{c}{A} + q_i\langle f_i, u_i\rangle \right| + (1-\lambda) \sqrt{\frac{c^2}{A^2} + cq_i\|u_i\|^2 + \frac{2cq_i}{A} Re \langle f_i, u_i\rangle}  \geq \frac{c}{A}, $ which is not true. Therefore,  $Re \langle f_i ,u_i \rangle < 0 , 1\leq i \leq N.$ On the other hand, taking  $U=\{u_i\}_{i=1}^N,$ we have  $\sum\limits_{1 \leq i \leq N} \langle f_i ,u_i \rangle = tr(\Theta_{U} \Theta_{F}^{*} ) = tr(\Theta_{F}^{*} \Theta_{U}) =0 ,$ as $\Theta_{F}^{*} \Theta_{U}$ is the zero operator.  This in turn gives $Re  \left(\sum\limits_{1 \leq i \leq N} \langle f_i ,u_i \rangle \right)  =0,$ which leads to a contradiction. Hence, $ S_{F}^{-1}F  \in \Delta_{F}^{(1)} . $\\
	
	Now, suppose $G =  \left\{\frac{1}{A}f_i + h_i\right\}_{i=1}^N \in \Delta_{F}^{(1)}.$ Then,  $\mathcal{A}_{\lambda P}^{(1)}(F,G) = \mathcal{A}_{\lambda P}^{(1)}(F,S_{F}^{-1}F).$ As earlier, we can show that $\lambda \left| \frac{c}{A} + q_i\langle f_i, h_i\rangle \right| + (1-\lambda)\sqrt{\frac{c^2}{A^2} + cq_i\|h_i\|^2 + \frac{2cq_i}{A} Re \langle f_i, h_i\rangle}  \leq \dfrac{c}{A}, $ and so $Re \langle f_i ,h_i \rangle \leq 0, \; 1 \leq i \leq N.$ Using the fact that $\sum\limits_{1 \leq i \leq N} \langle f_i ,h_i \rangle = 0,$ we have $Re\langle f_i ,h_i \rangle =0, \; 1 \leq i \leq N.$  Consequently,
	$$\lambda \left| \frac{c}{A} + i \; Im (q_i\langle f_i, h_i\rangle ) \right| + (1-\lambda)\sqrt{\frac{c^2}{A^2} + cq_i\|h_i\|^2 }  \leq \frac{c}{A}, \;\forall 1 \leq i \leq N.$$
	 Now, if $\|h_i\| > 0,$ then $\lambda \left| \frac{c}{A} + i \; Im (q_i\langle f_i, h_i\rangle ) \right| + (1-\lambda)\sqrt{\frac{c^2}{A^2} + cq_i\|h_i\|^2 } > \dfrac{c}{A}.$  So, we may conclude that $h_i =0,\forall 1 \leq i \leq N.$ In other words, $\Delta_{F}^{(1)} = \left\{ S_{F}^{-1}F \right\}. $
\end{proof}

\noindent On a different note, assuming $\Delta_{F}^{(1)}$ is non-empty for a tight frame $F = \{f_i\}_{i=1}^N,$ the following two theorems provide necessary as well as sufficient conditions for having the canonical dual frame $S_{F}^{-1}F$ as a 1-erasure $PA_{SO}OD$ of $F.$ Towards this end, we make use of the following notations. Let $M= \max \left\{ q_i\left\| f_i \right\|^2  : 1 \leq i \leq N \right\}, \Gamma_1 = \left\{i: 1 \leq i \leq N \; \text{and}\; q_i \left\|f_i \right\|^2 = M   \right\} $ and $\Gamma_2 = \{1,2,\ldots, N\} \setminus \Gamma_1.$ The $\text{span of} \left\{f_i : i \in \Gamma_j  \right\}$ shall be denoted  by $E_j,$ for $j=1,2.$

\begin{thm}\label{prop3point2}
	Let $F$ be a tight frame for $\mathcal{H}.$ Let  $\{q_i\}_{i=1}^N $ be a weight number sequence and  $0\le \lambda\le 1.$ Suppose  $\Delta_{F}^{(1)} \neq \phi.$ If $E_1 \cap E_2 = \{0\},$ then  $S_{F}^{-1}F \in \Delta_{F}^{(1)}.$
\end{thm}
\begin{proof}
	Let $F = \{f_i\}_{i=1}^N$ be a tight frame with  bound $A$ and  $ G  \in \Delta_{F}^{(1)}.$ Then,  $G = \left\{ \frac{1}{A} f_i +u_i \right\}_{i=1}^N,$ where $\sum \limits_{1 \leq i \leq N} \langle f,u_i \rangle f_i =0,$ for all $f \in \mathcal{H}.$ As $E_1 \cap E_2 = \{0\},$ we get $\sum\limits_{i \in \Gamma_1} \langle f,u_i \rangle f_i =0 = \sum\limits_{i \in \Gamma_2} \langle f,u_i \rangle f_i.$ Taking  $F' = \{f_i\}_{ i \in \Gamma_1}$ and $U' = \{u_i \}_{ i \in \Gamma_1 },$ we then have
	\begin{equation} \label{eqn3point4}
		\sum\limits_{i \in \Gamma_1} \langle f_i , u_i \rangle = tr (\Theta_{U'} \Theta_{F'}^{*}) = tr(\Theta_{F'}^{*} \Theta _{U'}) = 0.
	\end{equation}
	\noindent For $1 \leq i \leq N,$ consider
	\begin{align}\label{eqn3point5}
		q_i \left|\left\langle f_i,\frac{1}{A} f_{i} + u_i \right\rangle \right| \nonumber&\leq \lambda q_i\left|\left\langle f_i,\frac{1}{A} f_{i} + u_i \right\rangle \right| + (1-\lambda)q_i\|f_i\| \left\|\frac{1}{A} f_{i} + u_i\right\| \nonumber\\& \leq \max_{1 \leq i \leq N} \lambda q_i\left|\left\langle f_i,\frac{1}{A} f_{i} + u_i \right\rangle \right| + (1-\lambda)q_i\|f_i\| \left\|\frac{1}{A} f_{i} + u_i\right\|\nonumber \\&=\mathcal{A}_{\lambda P}^{(1)}(F,G)\nonumber \\& \leq \mathcal{A}_{\lambda P}^{(1)}(F, S_{F}^{-1}F)\nonumber \\&= \max_{1 \leq i \leq N}\frac{q_i \|f_i\|^2}{A}  \nonumber \\& = \frac{M}{A},
	\end{align}
	by making use of Cauchy-Schwarz inequality for the first inequality above. Therefore,
	\begin{equation}\label{eqn3point6}
		\left| \frac{M}{A} + q_i\langle f_i, u_i\rangle \right| \leq \frac{M}{A},\;\forall i \in \Gamma_1.
	\end{equation}
	If $M=0,$ then $\mathcal{A}_{\lambda P}^{(1)}(F,S_{F}^{-1}F) = \frac{M}{A} =0$ and so $\mathcal{A}_{\lambda P}^{(1)}(F) =0. $ This implies that $S_{F}^{-1} F \in \Delta_{F}^{(1)}.$ Suppose $M \neq 0.$ Then $q_i >0,\;\forall i \in \Gamma_1.$ From  (\ref{eqn3point6}), it follows that $Re \langle f_i,u_i \rangle \leq 0,\;\forall i \in \Gamma_1,$ for otherwise we will have a contradiction. Using (\ref{eqn3point4}), we conclude that $Re \langle f_i,u_i \rangle =0, \; \forall i \in \Gamma_1.$ Now,
	\begin{align*}
		\mathcal{A}_{\lambda P}^{(1)}(F,G)  & \geq \max\limits_{i \in \Gamma_1}  q_i \left(\lambda \left|\left\langle f_i,\frac{1}{A} f_{i} + u_i \right\rangle \right| +(1-\lambda) \|f_i\| \left\|\frac{1}{A} f_{i} + u_i\right\| \right) \\&= \max\limits_{i \in \Gamma_1} q_i \left(\lambda \left|\left\langle f_i,\frac{1}{A} f_{i} + u_i \right\rangle \right| + (1- \lambda)\|f_i\| \sqrt{\frac{1}{A^2}\|f_i\|^2 + \|u_i\|^2 + \frac{2}{A} Re \langle f_i,u_i \rangle }\,\right)  \\&=  \max\limits_{i \in \Gamma_1} q_i \left( \lambda \left|\frac{1}{A} \|f_i\|^2 + i \;Im\langle f_i,u_i \rangle \right| +(1-\lambda) \|f_i\|\sqrt{\frac{1}{A^2}\|f_i\|^2 + \|u_i\|^2} \right)  \\& \geq \max\limits_{i \in \Gamma_1} \frac{q_i \|f_i\|^2}{A} \\& = \frac{M}{A} \\&= \mathcal{A}_{\lambda P}^{(1)}(F,S_{F}^{-1}F).
	\end{align*}
	As $G \in \Delta_{F}^{(1)},$  we have $\mathcal{A}_{\lambda P}^{(1)}(F)=\mathcal{A}_{\lambda P}^{(1)}(F,G)=\mathcal{A}_{\lambda P}^{(1)}(F,S_{F}^{-1}F)$ and hence, $S_{F}^{-1}F \in \Delta_{F}^{(1)}.$
\end{proof}

Under additional conditions, more can be said about $\Delta_{F}^{(1)},$ which is discussed in the theorem below.

\begin{thm}\label{thm3point5}
	Let $F$ be a tight frame for $\mathcal{H}.$  Let  $\{q_i\}_{i=1}^N $ be a weight number sequence and  $0\le \lambda < 1.$  Suppose $\Delta_{F}^{(1)} \neq \phi$  and $M >0.$  Then, $\Delta_{F}^{(1)} = \{S_{F}^{-1}F\}$ if and only if  $E_1 \cap E_2 = \{0\}$ and  $\{f_i\}_{i \in \Gamma_2}$ is linearly independent.
\end{thm}
In order to prove the above theorem, we make use of the following lemma.

\begin{lem}\label{lem3point6}
	Let $\{f_1,\ldots,f_m\} \subset \mathcal{H}$ be linearly independent and $\alpha $ be a scalar. Then, there exists $h \in \mathcal{H}$ such that $\langle f_i,h \rangle = \alpha, \;\forall\,i=1,2,\ldots,m.$
\end{lem}
\begin{proof}
	For $1 \leq i \leq m,$ $f_i$ can be written as $f_i = f_{i1}e_1 + \cdots + f_{in}e_n,$ where $\{e_1,\ldots,e_n\}$ is an orthonormal basis for $\mathcal{H}.$ Consider the following system of $m$ equations in $n$ unknowns, namely, ${h}_1,{h}_2,\ldots,{h}_n.$
	\begin{align*}\label{eqn7}
		&f_{11} {h}_1 + \cdots +  f_{1n} {h}_n = \alpha\\
		&f_{21} {h}_1 + \cdots + f_{2n} {h}_n= \alpha\\
		&\qquad\qquad\vdots\\
		&f_{m1} {h}_1 + \cdots + f_{mn} {h}_n = \alpha.
	\end{align*}
	As $\{f_1,\ldots,f_m\}$ is linearly independent, the coefficient matrix of the above system has rank $m$ and so does the augmented matrix. Hence, the above system of equations has a solution and $h = \sum \limits_ {j=1}^n h_j e_j$ satisfies  $\langle f_i,h \rangle = \alpha \;\forall\,i=1,2,\ldots,m.$
\end{proof}
\begin{proof}[\textbf{Proof of Theorem \ref{thm3point5}}]
	Let $A$ denote the frame bound of $F.$  Assume that $E_1 \cap E_2 = \{0\}$ and  $\{f_i\}_{i \in \Gamma_2}$ is linearly independent. By Theorem \ref{prop3point2},  $S_{F}^{-1}F \in \Delta_{F}^{(1)}.$	Suppose $G= \left\{ \frac{1}{A}f_i +u_i \right\}_{i=1}^N \in \Delta_F^{(1)}.$ Using the linear independence of $\{f_i\}_{i \in \Gamma_2},$  we shall show that $u_i =0,\; 1 \leq i \leq N.$ Let $i \in \Gamma_1.$ From  (\ref{eqn3point5}), we have  $\lambda q_i\left|\left\langle f_i,\frac{1}{A} f_{i} + u_i \right\rangle \right| + (1-\lambda)q_i\|f_i\| \left\|\frac{1}{A} f_{i} + u_i\right\| \leq \frac{M}{A}.$ It is also shown in the proof of Theorem \ref{prop3point2} that $$ q_i \left( \lambda \left|\left\langle f_i,\frac{1}{A} f_{i} + u_i \right\rangle \right| + (1-\lambda)\|f_i\| \left\|\frac{1}{A} f_{i} + u_i\right\| \right) = q_i \left( \lambda \left|\frac{1}{A} \|f_i\|^2 + i \;Im\langle f_i,u_i \rangle \right| + (1-\lambda) \|f_i\|\sqrt{\frac{1}{A^2}\|f_i\|^2 + \|u_i\|^2} \right).$$
	Consequently, as $q_i\|f_i\|^2 =M,$ we obtain $ \lambda \left| \frac{M}{A} + i \;q_i Im \langle f_i,u_i \rangle \right| + (1- \lambda)\sqrt{\frac{M^2}{A^2} + q_{i}^2 \|f_i\|^2\|u_i\|^2}  \leq \frac{M}{A}.$ We note that $M>0$ implies that both $q_i$ and $\|f_i\|$ are positive. Further, if $\|u_i\| > 0,$ then $ \lambda \left| \frac{M}{A} + i \;q_i Im \langle f_i,u_i \rangle \right| + (1- \lambda)\sqrt{\frac{M^2}{A^2} + q_{i}^2 \|f_i\|^2\|u_i\|^2} > \frac{M}{A},$ which is absurd. Therefore, $u_i =0,\,\forall \, i \in \Gamma_1.$ Furthermore, from the proof of Theorem \ref{prop3point2}, we also have $\sum\limits_{i \in \Gamma_2}\langle f,u_i \rangle f_i =0,\,\forall \,f\in \mathcal{H}.$ The linear independence of  $\{f_i\}_{i \in \Gamma_2}$ then suggests that $u_i =0, \,\forall \, i \in \Gamma_2.$
	Conversely, assume that $\Delta_{F}^{(1)} = \{S_{F}^{-1}F\}.$ First, we shall prove that $\{f_i\}_{i\in \Gamma_2}$ is linearly independent. Suppose not. Then, there exist $\alpha_i,\, i \in \Gamma_2,$ not all zero, such that $\sum\limits_{i \in \Gamma_2} \alpha_i f_i =0.$ Let $f $ be a non-zero element in $\mathcal{H}$ and $U= \{u_i\}_{1 \leq i \leq N}$ be a sequence in $\mathcal{H},$ where  $u_i =0,\; i \in \Gamma_1$ and $u_i = \bar{\alpha_i}f,\,\forall i \in \Gamma_2.$  Clearly, for $ t \in \mathbb{R},$ $\sum\limits_{1 \leq i\leq N} \langle f,f_i \rangle tu_i=\sum\limits_{i \in \Gamma_2} \langle f,f_i \rangle tu_i = \left\langle f, \sum\limits_{i \in \Gamma_2} \alpha_i f_i \right\rangle tf  =0.$   So, $G =\left\{ \frac{1}{A}f_i +tu_i \right\}_{i=1}^N,\,t \in \mathbb{R}$ is a noncanonical dual of $F.$ Now, for $i \in \Gamma_1,\;  \lambda q_i \left| \left\langle f_i,\frac{1}{A}f_i +tu_i \right\rangle \right| + (1-\lambda)q_i \|f_i\|\left\|\frac{1}{A}f_i +tu_i \right\|  = \dfrac{M}{A}\;\forall\, t\in \mathbb{R}$ and for $ i \in \Gamma_2,\,\lambda q_i \left| \left\langle f_i,\frac{1}{A}f_i +tu_i \right\rangle \right| + (1-\lambda)q_i \|f_i\|\left\|\frac{1}{A}f_i +tu_i \right\|\to \dfrac{q_i \|f_i\|^2}{A} < \dfrac{M}{A} $ as $t \to 0.$ We shall choose $t_0 >0,$ small enough, such that $\lambda q_i \left| \left\langle f_i,\frac{1}{A}f_i +t_0 u_i \right\rangle \right| + (1-\lambda)q_i \|f_i\|\left\|\frac{1}{A}f_i +t_0 u_i \right\| < \dfrac{M}{A},\;\forall \, i\in \Gamma_2. $ Therefore, for the dual $G_0= \left\{\frac{1}{A}f_i + t_0 u_i  \right\}_{i=1}^N$ of $F$,  $\mathcal{A}_{\lambda P}^{(1)}(F,G_0) = \dfrac{M}{A} = \mathcal{A}_{\lambda P}^{(1)}(F,S_{F}^{-1}F)$  and thus we have arrived at a contradiction. Hence, $\{f_i\}_{i \in \Gamma_2}$ is linearly independent. \par
Next, we shall show that $E_1 \cap E_2 = \{0\}.$ Suppose $0 \neq f \in E_1 \cap E_2 .$ Then, $f$ can be written as $f= \sum\limits_{j \in \Gamma'_1 \subset \Gamma_1} c_{j} f_{j} = \sum\limits_{k \in \Gamma_2} d_k f_k, $ where $c_j \neq 0$ and $\{f_{j}\}_{j \in \Gamma_1'}$  is also linearly independent. As $\{ c_j f_j\}_{j \in \Gamma'_1}$ is also linearly  independent, by Lemma \ref{lem3point6}, there exists  $h\in \mathcal{H}$ such that $\langle  f_j, \bar{c}_j h \rangle <0,$ for all $ j \in \Gamma'_1.$ Now, let  $\tilde{U} =\{\tilde{u}_i\}_{i=1}^N,$ where
\[	\tilde{u}_i=\begin{cases}
	\bar{c}_i h, & \text{if $i \in \Gamma'_1$}\\
	0, & \text{if $i \in \Gamma_1 \setminus \Gamma'_1$}\\
	-\bar{d}_i h, & \text{if $i \in \Gamma_2$}.
\end{cases}
\]
Then for any $t\in \mathbb{R},\, \sum\limits_{1\leq i \leq N} \langle f,f_i \rangle t \Tilde{u}_i  =\sum\limits_{i \in \Gamma'_1} \langle f,c_if_i \rangle th - \sum\limits_{i \in \Gamma_2} \langle f,d_if_i \rangle th = \left\langle f, \sum\limits_{i \in \Gamma'_1}c_i f_i - \sum\limits_{i \in \Gamma_2} d_i f_i \right\rangle th =0.$ Therefore, $ \{\frac{1}{A}f_i + t\Tilde{u}_i\}_{i=1}^N$ is a noncanonical dual of $F,$ for any $t \in \mathbb{R}.$ As earlier, there exists $\tilde{t}_0 >0$ such that
$$  \lambda q_i \left| \left\langle f_i,\frac{1}{A}f_i + \tilde{t}_0 \Tilde{u}_i \right\rangle \right| + (1-\lambda)q_i \|f_i\|\left\|\frac{1}{A}f_i + \tilde{t}_0 \Tilde{u}_i \right\|  < \dfrac{M}{A}, \text{ for all } i \in \Gamma_2. $$
Further, for $t \in \mathbb{R},\,i \in \Gamma_1 \setminus \Gamma'_1 , \,$ we have $  \lambda q_i \left|\left\langle f_i, \frac{1}{A}f_i +t\Tilde{u}_i \right\rangle \right| + (1-\lambda)q_i\|f_i\| \left\|\frac{1}{A}f_i +t\Tilde{u}_i \right\|  = \dfrac{q_i \|f_i\|^2}{A} = \dfrac{M}{A}$ and for
$i \in \Gamma'_1 ,$
\small
\begin{align*}
 \lambda q_i \left|\left\langle\ f_i,\frac{1}{A}f_i +t\Tilde{u}_i \right\rangle \right| + (1-\lambda)q_i\|f_i\|\left\|\frac{1}{A}f_i +t\Tilde{u}_i \right\|  &=  \left(\lambda  \left| \frac{M}{A} + q_i t \langle f_i,\Tilde{u}_i \rangle   \right| + (1-\lambda)q_i\|f_i\| \sqrt{\frac{\|f_i\|^2}{A^2} + t^2\|\Tilde{u}_i\|^2 + \frac{2t}{A}Re\langle f_i,\Tilde{u}_i\rangle}\right) \\&= \lambda \left| \frac{M}{A} + q_it \langle f_i,\Tilde{u}_i \rangle   \right| + (1-\lambda)\sqrt{\frac{M^2}{A^2} + {q_i}^2 t^2 \|f_i\|^2 \|\Tilde{u}_i\|^2 +  \frac{2q_i tM \langle f_i,\Tilde{u}_i \rangle}{A} } ,
\end{align*}
as $\langle  f_i, \Tilde{u}_i \rangle <0.$
\normalsize
The negative value of $\langle  f_i, \Tilde{u}_i \rangle $ allows us to choose $t'_i >0$ small enough such that $\left| \dfrac{M}{A} + q_i t'_i \langle f_i,\Tilde{u}_i \rangle   \right| < \dfrac{M}{A}.$ We may also choose $t''_i >0$ small enough such that $q_i t''_i \|f_i\|^2\|\Tilde{u}_i\|^2 + \dfrac{2M \langle f_i,\Tilde{u}_i\rangle}{A}  <0,$ which implies that $\dfrac{M^2}{A^2} + {q_i}^2 {t''_i}^2 \|f_i\|^2 \|\Tilde{u}_i\|^2 +  \dfrac{2q_i t''_i M \langle f_i,\Tilde{u}_i \rangle}{A}  < \dfrac{M^2}{A^2}.$ Taking $t_i = \min\{t'_i,t''_i\}, $ we obtain  $\lambda q_i \left| \left\langle f_i,\frac{1}{A}f_i +t_i \Tilde{u}_i \right\rangle \right|+ (1-\lambda)q_i \|f_i\| \left\|\frac{1}{A}f_i + t_i \Tilde{u}_i \right\| < \dfrac{M}{A}.$ Furthermore, taking $t_1 = \min \{t_i : i \in \Gamma'_1\},$ we get   $$\lambda q_i \left|\left\langle f_i,\frac{1}{A}f_i +t_1 \Tilde{u}_i \right\rangle \right|+ (1-\lambda)q_i \|f_i\|\,\left\|\frac{1}{A}f_i +t_1 \Tilde{u}_i \right\|  < \frac{M}{A}, \text{ for all } i \in \Gamma'_1.$$
If we choose  $\Tilde{t} = min \{ \tilde{t}_0,t_1\},$ we obtain a dual $\Tilde{G}= \left\{ \frac{1}{A}f_i + \Tilde{t}\Tilde{u}_i\\\\\right\}_{i=1}^N$ of $F$ for which $\mathcal{A}_{\lambda P}^{(1)}(F,\Tilde{G}) = \frac{M}{A},$ thereby contradicting the hypothesis.
\end{proof}

\section{Relations between $POD,\,PSOD$ and $PA_{SO}OD$ }
In this section, we analyse certain relations between the three types of probabilistic optimal dual frames discussed so far, namely $POD,\,PSOD$ and $PA_{SO}OD.$

\begin{thm}\label{thm3point4}
	Let $F = \{f_i\}_{i=1}^N $ be a tight frame for $\mathcal{H}$. Let  $\{q_i\}_{i=1}^N $ be a weight number sequence and  $0\le \lambda < 1.$ Then the following are equivalent.
	
	\item (i) $S_{F}^{-1} F$ is a $1-$erasure $PA_{SO}OD$ of $F.$
	\item (ii) $S_{F}^{-1} F$ is a $1-$erasure $POD$ of $F.$
	\item (iii) $S_{F}^{-1} F$ is a $1-$erasure $PSOD$ of $F.$

\end{thm}

\begin{proof}
	Let $F$ be a tight frame with bound $A.$ First we shall prove the equivalence of (i) and (ii).
	\par Suppose the canonical dual $S_{F}^{-1} F = \left\{\frac{1}{A} f_{i}\right\}_{i=1}^N \in \Delta_{F}^{(1)}.$ Let $G= \{g_i\}_{i=1}^N$ be any dual of $F.$ Then, using \eqref{equationON5}, we get
	\begin{align*}\label{eqn4point1}
		\mathcal{O}_P^{(1)}(F,S_{F}^{-1} F)= \max_{1 \leq i \leq N} q_i\|f_i\| \left\|\tfrac{1}{A} f_{i}\right\| \nonumber&= \max_{1 \leq i \leq N} \bigg\{\lambda q_i|\langle f_i,\tfrac{1}{A} f_{i} \rangle | + (1-\lambda)q_i\|f_i\| \left\|\tfrac{1}{A} f_{i}\right\| \bigg\} \nonumber \\&= \mathcal{A}_{\lambda P}^{(1)}(F,\tfrac{1}{A}F) \nonumber \\&\leq \mathcal{A}_{\lambda P}^{(1)}(F,G) \nonumber \\&=  \max_{1 \leq i \leq N} \bigg\{\lambda q_i|\langle f_i,g_i \rangle | + (1-\lambda)q_i\|f_i\|\; \| g_{i}\| \bigg\} \nonumber \\& \leq \mathcal{O}_P^{(1)}(F,G)
	\end{align*}
	and so $S_{F}^{-1} F $ is a  $1-$erasure $POD$ of $F.$
	
	\par In order to prove the converse, we assume that  $S_{F}^{-1} F $ is a $1-$erasure $POD$ of $F$ but  $S_{F}^{-1} F \notin \Delta_{F}^{(1)}.$ Then, there exists a dual $G = \left\{ \tfrac{1}{A} f_{i} + u_i \right\}_{i=1}^N$ of $F$ such that $\mathcal{A}_{\lambda P}^{(1)}(F,G) < \mathcal{A}_{\lambda P}^{(1)}(F,S_{F}^{-1}F).$ Let $M, \Gamma_1$ and $\Gamma_2$ be as in Theorem \ref{prop3point2}. For $i \in \Gamma_1,$ we have
	\begin{align*}
		&\lambda \left| \frac{q_i}{A} \|f_{i}\|^2 + q_i\langle f_i, u_i\rangle \right| + (1-\lambda)\sqrt{\frac{q_i^2}{A^2}\|f_i\|^4 + q_i^2 \|f_i\|^2\|u_i\|^2 + \frac{2}{A}q_i^2\|f_i\|^2 Re \langle f_i, u_i\rangle}  \\&= \lambda q_i\left|\langle f_i,\frac{1}{A} f_{i} + u_i \rangle \right| + (1-\lambda)q_i\|f_i\|\; \left\|\tfrac{1}{A} f_{i} +u_i\right\|  \\& \leq \max\limits_{1 \leq j \leq N} \lambda q_j\left|\langle f_j,\frac{1}{A} f_{j} + u_j \rangle \right| + (1-\lambda)q_j\|f_j\|\; \left\|\tfrac{1}{A} f_{j} +u_j\right\|  \\&< \max\limits_{1 \leq j \leq N} \frac{q_j \|f_j\|^2}{A},
	\end{align*}
	which implies that \,  $\lambda \left| \frac{M}{A}  + q_i\langle f_i, u_i\rangle \right| + (1-\lambda)\sqrt{\frac{M^2}{A^2} + q_i M\|u_i\|^2 + \frac{2q_i}{A} M  Re \langle f_i, u_i\rangle} < \dfrac{M}{A}.$ From this inequality, we may conclude that $Re \langle f_i, u_i \rangle < 0.$ Consequently, for each  $i \in \Gamma_{1},$ we can choose $\epsilon_{1}^{(i)} > 0,$ small enough such that $\epsilon_{1}^{(i)}\|u_i\|^2 + \frac{2}{A} \text{Re}\langle f_i, u_i\rangle < 0.$ Taking $\epsilon_{1} \leq \min \limits_{i \in \Gamma_1}\epsilon_{1}^{(i)},$ we get for all  $i \in \Gamma_{1},$
	\begin{align*}
		q_i\|f_i\|\;\left\|\frac{1}{A}f_i + \epsilon_{1} u_i \right\| =q_i\sqrt{\frac{1}{A^2}\|f_i\|^4 + \epsilon_{1} \|f_i\|^2 \left(\epsilon_{1} \|u_i\|^2 + \frac{2}{A} Re \langle f_i, u_i\rangle \right)} \;\;<\; \frac{q_i \|f_i\|^2}{A} = \frac{M}{A}.
	\end{align*}

\noindent
Now, for each $j \in \Gamma_{2},\, M^2 - q_j^2 \|f_j\|^4 >0.$ Then, by the continuity of $q_j^2 \|f_j\|^2 t \left( t\|u_j\|^2 + \frac{2}{A} Re \langle f_j,u_j \rangle  \right)$ at $t=0,$ there exists  $\epsilon_{2}^{(j)} > 0$ such that $ q^2_j \| f_j\|^2 t \left(t\|u_j\|^2 + \frac{2}{A}  \text{Re} \langle f_j, u_j \rangle \right) < \frac{M^2 - q^2_j \|f_j\|^4}{A^2} ,\;\forall t < \epsilon_{2}^{(j)}. $ Then,
\begin{align*}
	q_j \|f_j\|\;\|\tfrac{1}{A}f_j + \epsilon_{2} u_j \| =\sqrt{\frac{q_j^{2}}{A^2}\|f_j\|^4 + q_j^{2} \|f_j\|^2 \epsilon_{2} \left(\epsilon_{2} \|u_j\|^2 + \frac{2}{A} \text{Re} \langle f_j, u_j \rangle \right)} \;<\; \frac{M}{A},
\end{align*}
where $\epsilon_{2} = \min\limits_{j \in \Gamma_2} \epsilon_{2}^{(j)}.$ Therefore, by taking $\epsilon < \displaystyle{\min \;\left\{\epsilon_{1}, \epsilon_{2} \right\}},$ we have  $\max\limits_{1 \leq i \leq N}  q_i \|f_i\| \; \left\|\tfrac{1}{A}f_i + \epsilon u_i \right\| < \frac{M}{A}.$ Hence, for the dual $G_{\epsilon} = \left\{\frac{1}{A}f_i + \epsilon u_i    \right\}_{i=1}^N,\;\;   \mathcal{O}_P^{(1)}(F,G_{\epsilon}) <   \mathcal{O}_P^{(1)}(F,S_{F}^{-1} F)$ holds, which is a contradiction to the fact that  $S_{F}^{-1} F $ is a $1-$erasure $POD$ of $F.$ Therefore, $S_{F}^{-1} F  \in  \Delta_{F}^{(1)}.$ \\

Next, we shall show the equivalence of (i) and (iii). If $S_{F}^{-1} F$ is a $1-$erasure $PSOD$ of $F,$ then for any dual  $G = \left\{\frac{1}{A}f_i +u_i\right\}_{i=1}^N$ of $F$ in $\mathcal{H},$
\begin{align*}
	\mathcal{A}_{\lambda P}^{(1)}(F,G) &= \max_{1 \leq i \leq N} \lambda q_i \left|\langle f_i,\tfrac{1}{A} f_{i} + u_i\rangle \right| + (1-\lambda)q_i \|f_i\|\; \left\|\tfrac{1}{A} f_{i} +u_i\right\|  \\& \geq  \max_{1 \leq i \leq N} \;q_i |\langle f_i,\tfrac{1}{A} f_{i} + u_i\rangle | \\&=  r_{P}^{(1)}(F,G) \\&\geq  r_{P}^{(1)}(F,S_{F}^{-1}F) \\&= \max_{1 \leq i \leq N} \frac{q_i \|f_i\|^2}{A}  \\&=  \mathcal{A}_{\lambda P}^{(1)}(F,S_{F}^{-1}F) .
\end{align*}
Hence, $S_{F}^{-1}F \in \Delta_{F}^{(1)}.$
\par Conversely, suppose $S_{F}^{-1} F \in \Delta_{F}^{(1)} $  and $G = \left\{\frac{1}{A}f_i +u_i\right\}_{i=1}^N$ is a noncanonical dual of $F$ such that $r_{P}^{(1)}(F,G) < r_{P}^{(1)}(F, S_{F}^{-1} F)$ i.e., $\max\limits_{1 \leq i \leq N} \;\; q_i \left|\langle f_i,\tfrac{1}{A} f_{i} +u_i \rangle \right| < \max\limits_{1 \leq i \leq N} \;\; q_i \left|\langle f_i,\tfrac{1}{A} f_{i} \rangle \right| .$ Let  $M ,\Gamma_{1}$ and $\Gamma_{2} $ be  as in Theorem \ref{prop3point2}. We then have for all $ i \in \Gamma_1,\;  q_i \left|\langle f_i,\tfrac{1}{A} f_{i} +u_i \rangle \right| < \frac{M}{A}.$ So, $  \left|\frac{q_i}{A}\|f_i\|^2 + q_i\langle f_i,u_i \rangle \right| < \frac{q_i}{A} \|f_i\|^2,$ which leads to the implication $Re \langle f_i , u_i \rangle <0.$ As in the proof of equivalence of (i) and (ii), we now choose $\epsilon >0 $ such that
\begin{align*}
	\max \limits_{1 \leq i \leq N} \bigg\{ \lambda q_i |\langle f_i,\tfrac{1}{A} f_{i} + \epsilon u_i\rangle | + (1-\lambda)q_i \|f_i\|\; \left\|\tfrac{1}{A} f_{i} + \epsilon u_i\right\| \bigg\}  \leq \max \limits_{1 \leq i \leq N} q_i \|f_i\|\; \left\|\tfrac{1}{A} f_{i} + \epsilon u_i\right\|  < \frac{M}{A} .
\end{align*}
Consequently, we have a dual, $G_{\epsilon}= \left\{\frac{1}{A}f_i + \epsilon u_i\right\}_{i=1}^N ,$ for which $ \mathcal{A}_{\lambda P}^{(1)}(F,G_{\epsilon}) < \mathcal{A}_{\lambda P}^{(1)}(F,S_{F}^{-1}F).$ Thus, we have arrived at a contradiction, thereby proving that $S_{F}^{-1} F$ is a $1-$erasure $PSOD$ of $F.$
\end{proof}

\begin{cor}
	Let $F$ be a tight frame for $\mathcal{H}.$ Let  $\{q_i\}_{i=1}^N $ be a weight number sequence and  $0\le \lambda\le 1.$ If $\Delta_{F}^{(1)} = \left\{S_{F}^{-1}F \right\},$ then $S_{F}^{-1}F$ is the only $1-$erasure $POD$ of $F.$
\end{cor}
\begin{proof}
	If $\Delta_{F}^{(1)} = \left\{S_{F}^{-1}F \right\},$ then, by the proof of the implication $(i)\Rightarrow (ii)$ of Theorem \ref{thm3point4}, $S_{F}^{-1}F$ is a $1-$erasure $POD$ of $F.$ Further, for any noncanonical dual $G$ of $F,$
	\begin{align*}
		\mathcal{O}_P^{(1)}(F,G) \geq \mathcal{A}_{\lambda P}^{(1)}(F,G) > \mathcal{A}_{\lambda P}^{(1)}(F,S_{F}^{-1}F) = \mathcal{O}_P^{(1)}(F,S_{F}^{-1}F),
	\end{align*}
	which shows that no noncanonical dual can be a $1-$erasure $POD$ of $F.$
	
\end{proof}

\noindent
In the case of a general frame, it is possible that the canonical dual frame is none of the optimal dual frames mentioned above, as shown by the example below.

\begin{example}
	Let $ \mathcal{H}= {\mathbb{R}}^2,$  frame $F= \{f_1,f_2,f_3\},$ where $f_1 = \left[\begin{array}{l}
		1 \\0
	\end{array}\right] $,
	$f_2 = \left[\begin{array}{l}
		0 \\ 1
	\end{array}\right] ,$
	$ f_3 = \left[\begin{array}{l}
		1 \\ 1
	\end{array}\right] $ and   $\{q_i\}_{i=1}^3 = \left\{ 2, \frac{3}{2}, \frac{6}{5} \right\}$ be the weight number sequence. It is easy to show that
	$S_F= \begin{bmatrix}
		2 & 1\\
		1 & 2
	\end{bmatrix} $ and the canonical dual is
	$$S_{F}^{-1} F = \left\{ \left[\begin{array}{r} \frac{2}{3} \\ - \frac{1}{3} \end{array}\right],  \left[\begin{array}{r} -\frac{1}{3} \\ \frac{2}{3} \end{array}\right], \left[\begin{array}{l} \frac{1}{3} \\ \frac{1}{3}  \end{array}\right] \right\} .$$
	\noindent We then compute the following:
	$q_1 \langle f_1 , S_{F}^{-1}f_1 \rangle = \frac{4}{3},q_2 \langle f_2 , S_{F}^{-1}f_2 \rangle = 1, q_3 \langle f_3 , S_{F}^{-1}f_3 \rangle = \frac{4}{5} \; \text{and}\; q_1 \|f_1\|\left\| S_{F}^{-1}f_1 \right\| = \frac{2\sqrt{5}}{3} ,\;q_2 \|f_2\|\left\| S_{F}^{-1}f_2 \right\| = \frac{\sqrt{5}}{2},\;q_3 \|f_3\|\left\| S_{F}^{-1}f_3\right\| = \frac{4}{5}.$ Therefore, $r_{P}^{(1)}( F, S_{F}^{-1}F) = \frac{4}{3}$, $\mathcal{O}_{P}^{(1)}( F, S_{F}^{-1}F) = \frac{2\sqrt{5}}{3}. $
	\begin{align*}
		\mathcal{A}_{\lambda P}^{(1)}( F, S_{F}^{-1}F) =\max \left\{2\left(\lambda \frac{2}{3} +(1-\lambda)\frac{\sqrt{5}}{3} \right), \frac{3}{2}\left(\lambda \frac{2}{3} +(1-\lambda)\frac{\sqrt{5}}{3} \right), \frac{6}{5}\left(\lambda \frac{2}{3} +(1-\lambda)\frac{2}{3} \right) \right\} = \lambda \frac{4}{3} +(1-\lambda)\frac{2\sqrt{5}}{3}
	\end{align*}
	
	 One can easily verify that the set of duals of $F$ is of the form \\
	$$D = \left\{ \left[\begin{array}{r} \frac{2}{3} + \alpha  \\ - \frac{1}{3} +\beta \end{array}\right], \left[\begin{array}{r} - \frac{1}{3} +\alpha  \\  \frac{2}{3} +\beta   \end{array}\right], \left[\begin{array}{l}\frac{1}{3} - \alpha \\ \frac{1}{3} - \beta \end{array}\right]  \right\}, \text{where}\; \alpha, \beta \in \mathbb{R}.$$
	Then, $r_{P}^{(1)}( F,G) = \max \left\{2 \left| \frac{2}{3} + \alpha \right|, \frac{3}{2}\left|\frac{2}{3} + \beta \right|, \frac{6}{5}\left|\frac{2}{3} - \alpha-\beta \right|\right\}.$ In particular, if we choose $\alpha = \beta = -\frac{1}{6},$ we deduce that $r_{P}^{(1)}( F,G) = \frac{6}{5}$ and hence, the canonical dual is not a $1-$erasure $PSOD$.
	Similarly, $\mathcal{O}_{P}^{(1)}( F,G) = \max \left\{ 2\sqrt{\left(\frac{2}{3} + \alpha\right)^2 + \left(\frac{1}{3} - \beta \right)^2}, \frac{3}{2}\sqrt{\left(\frac{1}{3} - \alpha\right)^2 + \left(\frac{2}{3} + \beta \right)^2}, \frac{6\sqrt{2}}{5}\sqrt{\left(\frac{1}{3} - \alpha \right)^2 + \left(\frac{1}{3} - \beta \right)^2}  \right\}$ and for the above choice of $\alpha$ and $\beta,$ we get $\mathcal{O}_{P}^{(1)}( F,G) = \sqrt{2}.$ So, the canonical dual is not a $1-$erasure $POD$.
	Also,
\begin{align*}
 \mathcal{A}_{\lambda P}^{(1)}( F,G) &= \max \left\{2\left(\lambda \frac{1}{2} +(1-\lambda)\frac{1}{\sqrt{2}} \right), \frac{3}{2}\left(\lambda \frac{1}{2} +(1-\lambda)\frac{1}{\sqrt{2}} \right), \frac{6}{5}\left(\lambda  +(1-\lambda)\right) \right\} \\ &= \left\{\begin{array}{lll}
\lambda +\sqrt{2}(1-\lambda) & \mbox{ if } & 0\le \lambda\le \frac{5\sqrt{2}-6}{5(\sqrt{2}-1)}\\
\frac{6}{5}& \text{ if } & \frac{5\sqrt{2}-6}{5(\sqrt{2}-1)}\le \lambda \le 1.
\end{array}
\right.\\
\end{align*}

	Therefore, $\lambda \frac{4}{3} +(1-\lambda)\frac{2\sqrt{5}}{3} > \lambda +\sqrt{2}(1-\lambda)$ for $ 0\le \lambda\le \frac{5\sqrt{2}-6}{5(\sqrt{2}-1)}$  and hence the canonical dual is not a $1-$erasure $PA_{SO}OD$ either.
	
\end{example}

\noindent We now derive another relation between $PSOD$ and $PA_{SO}OD.$
\begin{thm}
	For a frame $F=\{f_i\}_{i=1}^N$ in $\mathcal{H}.$  Let  $\{q_i\}_{i=1}^N $ be a weight number sequence and  $0\le \lambda\le 1.$ Suppose $S_{F}^{-1}F$ is a $1-$erasure $PSOD$. Then, $S_{F}^{-1/2}F \in \Delta_{S_{F}^{-1/2}F}^{(1)}.$
\end{thm}

\begin{proof}
	Any dual $G = \{g_i\}_{i=1}^N$ of $S_{F}^{-1/2}F$ is of the form $\left\{S_{F}^{-1/2}f_i + u_i\right\}_{i=1}^N$ with $\sum \limits_{i=1}^N \langle f, S_{F}^{-1/2}f_i \rangle u_i  = 0,$ for all $f \in \mathcal{H},$ as $S_{F}^{-1/2}F$ is a Parseval frame for $\mathcal{H}.$ Now, consider
	\begin{align*}
		\mathcal{A}_{\lambda P}^{(1)}(S_{F}^{-1/2}F,G) &= \max_{1 \leq i \leq N}  \lambda q_i \left| \left\langle S_{F}^{-1/2}f_i,g_i \right\rangle \right| + (1-\lambda)q_i \left\| S_{F}^{-1/2}f_i \right\|\;\|g_i \| \\& \geq \max_{1 \leq i \leq N} q_i \left| \langle S_{F}^{-1/2}f_i,g_i \rangle \right| \\&= \max_{1 \leq i \leq N} q_i \left| \langle S_{F}^{-1/2}f_i, S_{F}^{-1/2}f_i +u_i \rangle \right| \\& = \max_{1 \leq i \leq N} q_i \left| \langle f_i, S_{F}^{-1}f_i + S_{F}^{-1/2}u_i \rangle \right| \\&= r_{P}^{(1)} (F,G'),
	\end{align*}
	where $G'= \left\{ S_{F}^{-1}f_i+S_{F}^{-1/2}u_i\right\}_{i=1}^N .$ We note that $G'$ is also a dual frame of $F$ for, $\sum \limits_{i=1}^N \langle f, f_i \rangle S_{F}^{-1/2}u_i = S_{F}^{-1/2} \left(\sum \limits_{i=1}^N \langle S_{F}^{1/2} f, S_{F}^{-1/2} f_i \rangle u_i\right) = 0,$ \,$f \in \mathcal{H}.$  As $S_{F}^{-1}F$ is a $1-$erasure $PSOD$ of $F,$ we obtain $r_{P}^{(1)} (F,G') \geq r_{P}^{(1)} (F,S_{F}^{-1}F)$ and hence,
	\begin{align*}
		\mathcal{A}_{\lambda P}^{(1)}(S_{F}^{-1/2}F,G)  \geq  r_{P}^{(1)} (F,S_{F}^{-1}F)= \max_{1 \leq i \leq N} q_i \| S_{F}^{-1/2}f_i\|^2 = \mathcal{A}_{\lambda P}^{(1)}(S_{F}^{-1/2}F,S_{F}^{-1/2}F) .
	\end{align*}
	Therefore, $S_{F}^{-1/2}F \in \Delta_{S_{F}^{-1/2}F}^{(1)}.$
\end{proof}

\section{Topological properties of $\Delta_{F}^{(1)}$}

Now, we shall look into some topological properties such as convexity, closedness and compactness of $\Delta_{F}^{(1)}.$   On $\mathcal{H}^{N},$ one can define several norms, which in fact are all equivalent. For $F\in \mathcal{H}^{N}, $ we take $\|F\|$ to denote the 2-norm, $\bigg(  \sum\limits_{i=1}^N \|f_i  \|^2 \bigg)^{\frac{1}{2}}.$

\begin{thm}
	Let $F = \{f_i\}_{i=1}^N $ be a frame for $\mathcal{H}$.  Let  $\{q_i\}_{i=1}^N $ be a weight number sequence and  $0\le \lambda\le 1.$ Then, the set $\Delta_{F}^{(1)} $ is a closed convex subset of $\mathcal{H}^N$.
\end{thm}
\begin{proof}
		Let $\{G^{(n)}\}_{n \in \mathbb{N}}$ be a sequence in $\Delta_{F}^{(1)}$ which converges to $G \in \mathcal{H}^{N}.$ 	In order to prove that $G =\{g_i\}_{i=1}^N$ is a dual frame of $F,$ it is enough to show that $\displaystyle{\sum_{i=1}^N \langle f,g_i \rangle f_i = f}, \forall f \in \mathcal{H},$ which can be verified easily. Now, $\mathcal{A}_{\lambda P}^{(1)} (F) = \mathcal{A}_{\lambda P}^{(1)} (F,G^{(n)}) = \max\limits_{1 \leq i \leq N} {q_i\bigg( \lambda | \langle f_i,g^{(n)}_i \rangle | + (1-\lambda) \|f_i\|\; \|g^{(n)}_i\|\bigg)}. $ Taking the limit as $n \to \infty,$ we get
	\begin{align*}
		\mathcal{A}_{\lambda P}^{(1)} (F) = \lim\limits_{n \to \infty}\max\limits_{1 \leq i \leq N} q_i\bigg(\lambda | \langle f_i,g^{(n)}_i \rangle | + (1-\lambda)\|f_i\|\; \|g^{(n)}_i\|\bigg) = \max\limits_{1 \leq i \leq N} q_i\bigg(\lambda | \langle f_i,g_i \rangle | + (1-\lambda)\|f_i\|\; \|g_i\|\bigg) = \mathcal{A}_{\lambda P}^{(1)} (F,G).
	\end{align*}
	Hence, $G \in \Delta_{F}^{(1)}$ and the set  $\Delta_{F}^{(1)}$ is closed.
	
	\par Suppose $ G= \{g_i\}_{i=1}^N,\;G'= \{g'_i\}_{i=1}^N \in \Delta_{F}^{(1)} .$  Let $\delta \in [0,1]$ and  $G'' = \delta G + (1-\delta) G'.$ It can be easily verified that $G''$ is a dual of $F.$	Further,
	\begin{align*}
		\mathcal{A}_{\lambda P}^{(1)}(F,G'') &= \max_{1 \leq i \leq N} \;\; q_i\bigg({ \lambda | \langle f_i,  \delta g_i +(1- \delta) g'_{i} \rangle | + (1-\lambda)\|f_i\|\; \| \delta g_i +(1- \delta) g'_{i}\|\bigg)} \\ & \leq \max_{1 \leq i \leq N} \;\; \left(\delta\; q_i\bigg({\lambda | \langle f_i, g_i  \rangle | + (1-\lambda)\|f_i\|\; \| g_i \|}\bigg) + (1-\delta)\; q_i\bigg({\lambda| \langle f_i, g'_i  \rangle | + (1-\lambda)\|f_i\|\; \| g'_i \|}\bigg) \right)  \\& \leq \delta \mathcal{A}_{\lambda P}^{(1)}(F,G) + (1-\delta)\mathcal{A}_{\lambda P}^{(1)}(F,G') \\&= \delta \mathcal{A}_{\lambda P}^{(1)}(F) + (1-\delta)\mathcal{A}_{\lambda P}^{(1)}(F) \\&= \mathcal{A}_{\lambda P}^{(1)}(F).
	\end{align*}
	Therefore,  $G'' \in \Delta_{F}^{(1)} $ and hence $\Delta_{F}^{(1)} $ is a convex set.
\end{proof}

\noindent In order to prove the compactness of $\Delta_{F}^{(1)},$ we make use of the following lemmas.

\begin{lem}\label{lemma3point1}
	Let $A=\{a_1,a_2,\cdots,a_N\}$ and $B=\{b_1,b_2,\cdots,b_N\}$ be two sets of real numbers. Then,\; $\left|\max\; A - \max\; B  \right| \leq \max\limits_{1 \leq k \leq N} \;|a_k - b_k|.$
\end{lem}
\begin{proof}
	Suppose $max\; A = a_{i_0}$ and $max\; B = b_{j_0},$ $1 \leq i_0,j_0 \leq N.$ Then,  $\left|max\; A - max\; B  \right| = |a_{i_0} - b_{j_0}|.$ If $a_{i_0} \geq b_{j_0},$ then $0 \leq a_{i_0} -b_{j_0} \leq a_{i_0} - b_{i_0}.$ Consequently, $|a_{i_0} -b_{j_0}|  \leq  \left| a_{i_0} -b_{i_0} \right| \leq \max\limits_{1 \leq k \leq N} | a_k - b_k|.$ On the other hand, if $a_{i_0} \leq b_{j_0},$ then $0 \leq b_{j_0} - a_{i_0} \leq b_{j_0} - a_{j_0}$ and so, $|a_{i_0} -b_{j_0}| = b_{j_0} - a_{i_0} \leq b_{j_0} - a_{j_0} = \left| a_{j_0} - b_{j_0} \right| \leq \max\limits_{1 \leq k \leq N} | a_k - b_k|,$ thereby proving the relation.
\end{proof}

\begin{lem}\label{lemma3point2}
	Let $F = \{f_i\}_{i=1}^N $ be a frame for $\mathcal{H}$ and $\mathcal{G}$ be the collection of all dual frames of $F$.  Let  $\{q_i\}_{i=1}^N $ be a weight number sequence and  $0\le \lambda\le 1.$ Then, the map $ \eta : \mathcal{G} \rightarrow \mathbb{R}^{+} \cup \{0\}\; \text{defined by}\; \eta(G) = \mathcal{A}_{\lambda P}^{(1)}(F,G)$ is  uniformly continuous.
\end{lem}
\begin{proof}
	Let $G'=\{g'_i\}_{i=1}^N,G''=\{g''_i\}_{i=1}^N \in \mathcal{G}.$ By Proposition \ref{prop3point1} and Lemma \ref{lemma3point1}, we have
	\begin{align*}
		|\eta(G') - \eta(G'') | &= \left|  \max_{1\leq i \leq N}\;\;q_i\bigg({\lambda | \langle f_i, g'_i  \rangle | + (1-\lambda)\|f_i\|\; \| g'_i \|}\bigg)  -  \max_{1\leq i \leq N}\;\;q_i\bigg({\lambda| \langle f_i, g''_i  \rangle | +(1-\lambda) \|f_i\|\; \| g''_i \|}\bigg) \right| \\& \leq \max_{1 \leq i \leq N} \left| q_i\bigg({\lambda | \langle f_i, g'_i  \rangle | + (1-\lambda)\|f_i\|\; \| g'_i \|}\bigg) - q_i\bigg({\lambda| \langle f_i, g''_i  \rangle | + (1-\lambda)\|f_i\|\; \| g''_i \|}\bigg)  \right| \\& \leq \max_{1 \leq i \leq N}\;\;q_i \bigg( \lambda \left| \;|\langle f_i, g'_i \rangle | -  |\langle f_i, g''_i \rangle | \; \right| + (1-\lambda ) \|f_i\| \left| \|g'_i\|-\|g''_i\| \right| \bigg) \\& \leq \max_{1 \leq i \leq N}\;\;q_i \bigg( \lambda |\langle f_i,g'_i - g''_i \rangle| + (1-\lambda)\|f_i\|\;\|g'_i - g''_i\|    \bigg) \\& \leq \max_{1 \leq i \leq N}\;\; q_i \|f_i\|\;\|g'_i - g''_i\| \\& \leq B \;\|G' -G''\|, \end{align*}
	where $B= \max\limits_{1 \leq i \leq N} q_i \|f_i\|.$ Hence, $\eta$ is a uniformly continuous function.
\end{proof}

\begin{thm}
	Let $F = \{f_i\}_{i=1}^N $ be a frame for $\mathcal{H}$ with $f_i \neq 0\,\forall i$.  Let  $\{q_i\}_{i=1}^N $ be a weight number sequence and $0\le \lambda < 1.$    Then, the set $\Delta_{F}^{(1)} $ is a nonempty compact subset of $\mathcal{H}^N$.
\end{thm}

\begin{proof}
	Let $\mathcal{G}$ denote the collection of all dual frames of $F$ in  $\mathcal{H}.$ Clearly, $\mathcal{G}$ is a closed subset of $\mathcal{H}^N.$  Consider the continuous map $\eta : \mathcal{G} \rightarrow \mathbb{R}^{+} \cup \{0\},$ defined as in Lemma \ref{lemma3point2}. Let $a := \eta\left( S_{F}^{-1}F  \right).$ We shall first show that $\eta^{-1}([0,a])$ is a compact subset of $\mathcal{H}^N.$ Clearly, $\eta^{-1}([0,a])$ is closed. Further, for any $G \in \eta^{-1}\left([0,a]\right),\,$ $\eta(G) \geq \max\limits_{1 \leq i \leq N} {q_i}(1-\lambda)\|f_i\|\,\|g_i\| \geq cd(1-\lambda) \max\limits_{1 \leq i \leq N}\|g_i\|,$ where $c= \min\limits_{1 \leq i \leq N} {q_i} $ and $d= \min\limits_{1 \leq i \leq N} {\|f_i\|}.$ So, $\eta(G) \geq \frac{cd(1-\lambda)}{N}\sum\limits_{i=1}^N \|g_i\| \geq \frac{cd(1-\lambda)}{N} \|G\|.$ This then implies that $\|G\| \leq \frac{Na}{cd(1-\lambda)},\,\forall G \in \eta^{-1}([0,a]),$ which proves that $\eta^{-1}([0,a])$ is compact. Therefore, there exists  $G' \in \eta^{-1}([0,a])$ such that $\eta(G') = b,$ where $b := \inf\limits_{G \in \eta^{-1}([0,a]) } \eta(G).$ In fact, it is true that $b = \inf\limits_{G \in \mathcal{G}} \eta(G) $ as well. Thus, $G' \in \Delta_{F}^{(1)},$ thereby proving that $\Delta_{F}^{(1)}$ is nonempty. Moreover, $\Delta_{F}^{(1)} = \eta^{-1}(b),$ which being a closed subset of the compact set $\eta^{-1}([0,a]),$ is also compact.
\end{proof}

\noindent
In the next theorem, we show that the image of a $1-$erasure $PA_{SO}OD$ of $F$ under a unitary operator is a $1-$erasure $PA_{SO}OD$ of the unitary image of $F.$
\begin{thm}
	Let $F = \{f_i\}_{i=1}^N $ be a frame for $\mathcal{H}$ and $U$ be a unitary operator on $\mathcal{H}.$  Let  $\{q_i\}_{i=1}^N $ be a weight number sequence and  $0\le \lambda\le 1.$  Then, $G \in \Delta_{F}^{(1)}$ if and only if $UG \in \Delta_{UF}^{(1)}.$
\end{thm}
\begin{proof}
	If $G=\{g_i\}_{i=1}^N$ is a dual frame for the frame $F,$ then $UG =  \{Ug_i\}_{i=1}^N$ is a dual frame for the frame $UF = \{Uf_i\}_{i=1}^N.$  Suppose $G \in \Delta_{F}^{(1)}.$ Let $G'= \{g'_i\}_{i=1}^N$ be a dual of $UF$. Then, $U^{*}G'$ is a dual of $F.$ \\
	Now,
	\begin{align*}
		\mathcal{A}_{\lambda P}^{(1)}(UF,UG) &= \max_{1 \leq i \leq N}  {q_i \bigg( \lambda | \langle Uf_i,Ug_i\rangle | + (1-\lambda)\|Uf_i\|\; \|Ug_i \|\bigg)}  \\&= \max\limits_{1 \leq i \leq N} {q_i \bigg( \lambda | \langle f_i,g_i\rangle | + (1-\lambda)\|f_i\|\; \|g_i \|\bigg)}  \\ &\leq \mathcal{A}_{\lambda P}^{(1)}(F,U^{*}G') \\ &= \max_{1 \leq i \leq N}  {q_i \bigg(\lambda| \langle f_i,U^*g'_i\rangle | + (1-\lambda)\|f_i\|\; \|U^*g'_i \|\bigg)} \\&= \max_{1 \leq i \leq N} {q_i \bigg(\lambda| \langle U f_i,g'_i\rangle | + (1-\lambda)\|U f_i\|\; \|g'_i \|\bigg)} \\&= \mathcal{A}_{\lambda P}^{(1)}(UF,G')
	\end{align*}
	Hence, $UG \in \Delta_{UF}^{(1)}.$ By taking $UF,UG\; \text{and}\; U^*$ in the place of $F,G \; \text{and}\; U$ respectively in the above argument, we can conclude that $UG \in \Delta_{UF}^{(1)}$ implies that $G \in \Delta_{F}^{(1)}.$
\end{proof}

\subsection*{Acknowledgment} S. Arati acknowledges the financial support of National Board for Higher Mathematics, Department of Atomic Energy, Government of India.

\bibliographystyle{amsplain}


\end{document}